\documentclass[12pt,letterpaper,final,twoside,leqno]{amsart}
\usepackage{amsmath,amsthm,amssymb,amsfonts,amscd,amsopn} %
\usepackage{eucal,mathrsfs,relsize}
\usepackage{xspace,chngcntr}
\usepackage{rotating,mathtools}
\usepackage[all,cmtip]{xy}\xyoption{dvips}
\usepackage{tikz-cd}
\usetikzlibrary{arrows.meta}
 
\usepackage{comment} 
\usepackage{supertabular}
\usepackage%[colorlinks=true%hidelinks
{hyperref}
\hypersetup{
    colorlinks=true,
    linkcolor={blue!50!black},
    citecolor={green!50!black},
    urlcolor={red!80!black}
}

\usepackage{datetime,chngcntr}
\usepackage{paralist}

\usepackage{enumitem}
\usepackage{eurosym}

\usepackage{mfirstuc}
\DeclareMathAlphabet{\smallchanc}{OT1}{pzc}%
                                 {m}{it}
\DeclareFontFamily{OT1}{pzc}{}
\DeclareFontShape{OT1}{pzc}{m}{it}%
             {<-> s * [1.100] pzcmi7t}{}
\DeclareMathAlphabet{\mathchanc}{OT1}{pzc}%
                                 {m}{it}
   [2017/06/13 v0.1 finetuning mabliautoref and theorem setup]

\RequirePackage{hyperref}
\newcommand{\sectionsize}{\relsize{-.5}}% \small} %\footnotesize}
\newcommand{\theoremsize}{\relsize{-.5}}
\newcommand{\mrsize}{\relsize{-.8}}
\renewcommand{\subsectionautorefname}{\sectionsize\sf \subsectionautorefname}

\makeatletter
\@ifdefinable\equationname{\let\equationname\equationautorefname}
\def\equationautorefname~#1\@empty\@empty\null{\protect{\theoremsize\sf%\relsize{-.5}
    (#1\@empty\@empty\null)}}% 
\@ifdefinable\AMSname{\let\AMSname\AMSautorefname}
\def\AMSautorefname~#1\@empty\@empty\null{\sf( #1\@empty\@empty\null)}%
\@ifdefinable\itemname{\let\itemname\itemautorefname}
\def\itemautorefname~#1\@empty\@empty\null{\mrsize{%
    {\sf #1}}\@empty\@empty\null%
}%
\makeatother

\RequirePackage{amsthm}
\RequirePackage{aliascnt}
\newcommand{\basetheorem}[3]{%
    \newtheorem{#1}{#2}[#3]
    \newtheorem*{#1*}{#2}
    \expandafter\def\csname #1autorefname\endcsname{#2}
}%
\newcommand{\maketheorem}[3]{%
    \newaliascnt{#1}{#2}
    \newtheorem{#1}[#1]{\theoremsize\sf #3}
    \aliascntresetthe{#1}
    \expandafter\def\csname #1autorefname\endcsname{\theoremsize\sf #3}
    \newtheorem{#1*}{#3}
}%
\newcommand{\baseremark}[3]{%
    \newtheorem{#1}{#2}{#3}
    \newtheorem*{#1*}{#2}
    \expandafter\def\csname #1autorefname\endcsname{#2}
}%
\newcommand{\makeremark}[3]{%
    \newaliascnt{#1}{#2}
    \newtheorem{#1}[equation]{#3}
    \aliascntresetthe{#1}
    \expandafter\def\csname #1autorefname\endcsname{\theoremsize\sf #3}
    \newtheorem{#1*}{#3}
}%
\theoremstyle{plain}   %-------------------standard Style-------------------------

\basetheorem{theorem}{Theorem}{section}
\newcommand{\mypagesize}{
\textwidth= 6.5in
\textheight=8.75in
\voffset-.5in
\hoffset-.75in
\marginparwidth=56pt
\footskip.5in
}

\newcounter{are-there-sections}
\newcommand{\zzzzz}{}
\mypagesize
\usepackage{amsbsy}

\usepackage{marvosym}

\DeclareMathAlphabet{\smallchanc}{OT1}{pzc}%
                                 {m}{it}
\DeclareFontFamily{OT1}{pzc}{}
\DeclareFontShape{OT1}{pzc}{m}{it}%
             {<-> s * [1.100] pzcmi7t}{}
\DeclareMathAlphabet{\mathchanc}{OT1}{pzc}%
                                 {m}{it}
\newcommand{\mcA}{\mathchanc{A}}

\newcommand{\mcD}{\mathchanc{D}}

\newcommand{\mcL}{\mathchanc{L}}

\newcommand{\mcR}{\mathchanc{R}}

\newcommand{\sO}{\mathscr{O}}

\newcommand{\sQ}{\mathscr{Q}}

\newcommand{\sfA}{{\sf A}}
\newcommand{\sfB}{{\sf B}}

\newcommand{\sfM}{{\sf M}}

\newcommand{\sfh}{{\sf h}}

\newcommand{\bC}{\mathbb{C}}

\newcommand{\bH}{\mathbb{H}}

\newcommand{\bN}{\mathbb{N}}

\newcommand{\bQ}{\mathbb{Q}}
\newcommand{\bR}{\mathbb{R}}

\newcommand{\bZ}{\mathbb{Z}}

\def\frm{\mathfrak{m}}
\newcommand{\frn}{\mathfrak{n}}

\DeclareSymbolFont{largesymbolsA}{U}{jkpexa}{m}{n}
\SetSymbolFont{largesymbolsA}{bold}{U}{jkpexa}{bx}{n}
\DeclareMathSymbol{\varprod}{\mathop}{largesymbolsA}{16}
\makeatletter
\newcommand{\LeftEqNo}{\let\veqno\@@leqno}
\makeatother
\newcommand{\ol}{\overline}

\newcommand{\into}{\hookrightarrow}

\newcommand{\onto}{\twoheadrightarrow}

\newcommand{\properideal}%
        {\subsetneq}
\newcommand{\wt}{\widetilde}
\newcommand{\what}{\widehat}

\newcommand{\rdown}[1]{\lfloor{#1}\rfloor}

\newcommand{\leteq}{\colon\!\!\!=}

\newcommand\dash[1]{\rule[-.2ex]{#1}{.4pt}}
\DeclareMathOperator{\ann}{\mathchanc{ann}}

\DeclareMathOperator{\kar}{char}

\DeclareMathOperator{\codim}{codim}

\DeclareMathOperator{\depth}{{depth}}

\newcommand{\sExt}[0]{{\mathchanc{Ext}}}

\DeclareMathOperator{\Hom}{Hom}
\newcommand{\sHom}[0]{{\mathchanc{Hom}}}

\DeclareMathOperator{\im}{{im}}

\DeclareMathOperator{\length}{{length}}

\newcommand{\lotimes}{\overset{L}{\otimes}}

\newcommand{\hotimes}[0]{%
  \protect{\ensuremath{\kern.1em{{%
  \raisebox{.5\depth}{$\scriptstyle[$}
  }}\kern-.25em\otimes\kern-.25em{%
  \raisebox{.5\depth}{$\scriptstyle]$}
  }\kern.1em}}}

\DeclareMathOperator{\Ob}{{Ob}}

\DeclareMathOperator{\pic}{{Pic}}

\newcommand{\red}{\mathrm{red}}

\DeclareMathOperator{\Spec}{{Spec}}

\DeclareMathOperator{\Soc}{{Soc}}
\DeclareMathOperator{\supp}{{supp}}

\DeclareMathOperator{\skvert}{{\,\vert\,}}

\newcommand{\factor}[2]{\left. \raise .2em\hbox{\ensuremath{#1}\vphantom{$I^d$}}
\hskip -.1em \right/ \hskip -.4em \raise -.3em\hbox{\ensuremath{#2}}}%
\newcommand\mtimes[3]{{\varprod_{#1}^{#2}}_{\raise 1ex \hbox{\scriptsize #3}}}%
\newcommand{\myR}{{\mcR\!}}
\newcommand{\myL}{{\mcL\!}}

\newcommand{\blank}{\dash{1em}}

\newcommand{\kdot}{{{\,\begin{picture}(1,1)(-1,-2)\circle*{2}\end{picture}\,}}}

\newcommand{\cmx}[1]{{#1}^{\raisebox{.15em}{\ensuremath\kdot}}}

\newcommand{\dcx}[1]{{\omega}^\kdot_{#1}}

\newcommand{\Om}{\underline{\Omega}}

\def\dimcoh#1.#2.#3.{h^{#1}(#2,#3)}
\def\hypcoh#1.#2.#3.{\mathbb H_{\vphantom{l}}^{#1}(#2,#3)}
\def\loccoh#1.#2.#3.#4.{H^{#1}_{#2}(#3,#4)}
\def\dimloccoh#1.#2.#3.#4.{h^{#1}_{#2}(#3,#4)}
\def\lochypcoh#1.#2.#3.#4.{\mathbb H^{#1}_{#2}(#3,#4)}
\def\seslong#1.#2.#3.{0  \longrightarrow  #1   \longrightarrow 
 #2 \longrightarrow #3 \longrightarrow 0} 
\def\sesshort#1.#2.#3.{0
 \rightarrow #1 \rightarrow #2 \rightarrow #3 \rightarrow 0}
\def\dist#1.#2.#3.{  #1   \longrightarrow 
 #2 \longrightarrow #3 \stackrel{+1}{\longrightarrow} } % \tag{$\bigtriangleup$}}
\def\CDdist#1.#2.#3.{  #1   @>>>  #2  @>>>   #3 @>+1>> }  
\def\shortses#1.#2.#3.{0  \rightarrow  #1   \rightarrow 
 #2  \rightarrow   #3 \rightarrow  0}
\def\shortdist#1.#2.#3.{  #1   \rightarrow 
 #2  \rightarrow   #3 \stackrel{+1}{\rightarrow} }  % \tag{$\bigtriangleup$}}
\def\ddist#1.#2.#3.#4.#5.#6.{\CD
#1 @>>> #2 @>>> #3 @>+1>> \\
@VVV @VVV @VVV \\
#4 @>>> #5 @>>> #6 @>+1>> 
\endCD}
\def\ddistun#1.#2.#3.#4.#5.#6.{\CD
#1 @>>> #2 @>>> #3 @>+1>> \\
@. @VVV @VVV  \\
#4 @>>> #5 @>>> #6 @>+1>> 
\endCD}
\def\Iff#1#2#3{
\hfil\hbox{\hsize =#1
\vtop{\noin #2}
\hskip.5cm 
\lower.5\baselineskip\hbox{$\Leftrightarrow$}\hskip.5cm
\vtop{\noin #3}}\hfil\medskip}
\newcommand{\union}\cup
\newcommand{\intersect}\cap
\newcommand{\Union}\bigcup
\newcommand{\Intersect}\bigcap
\def\myoplus#1.#2.{\underset #1 \to {\overset #2 \to \oplus}}

\newcommand{\resto}[1]{\raise -.5ex\hbox{$\vert$}_{#1}}%\vphantom{\vert}}}

\def\qis{\,{\simeq}_{\text{qis}}\,}
\newcommand{\qtq}[1]{\quad\mbox{#1}\quad}

\newcommand{\tsum}[0]{\textstyle{\sum}}

\newcommand{\ses}{short exact sequence\xspace}

\newcommand\dbcx[1]{\underline\Omega_{#1}^0}

\newcommand{\DB}{Du~Bois\xspace}

\newcommand{\sings}{singularities\xspace}

\begin{document}
\makeatletter
\definecolor{brick}{RGB}{204,0,0}
\def\@cite#1#2{{%
 \m@th\upshape\mdseries[{\sffamily\relsize{-.5}#1}{\if@tempswa,
   \sffamily\relsize{-.5}\color{brick} #2\fi}]}}
\newcommand{\sandor}{{\color{blue}{S\'andor \mdyydate\today}}}
\newenvironment{refmr}{}{}
\newcommand{\biblio}{%
\bibliographystyle{/home/kovacs/tex/TeX_input/skalpha} %$ 
\bibliography{/home/kovacs/tex/TeX_input/Ref} %$
}
% 
%\renewcommand{\labelenumi}{{\rm (\thethm.\arabic{enumi})}}
%\renewcommand{\labelenumi}{\hskip .5em(\thethm.\arabic{enumi})}
%%%%%%%%%%%% 
\definecolor{refblue}{RGB}{0,0,128}
        \newcommand\refblue{\color{refblue}}
\definecolor{mydarkblue}{RGB}{10,92,153}
        \newcommand\mydarkblue{\color{mydarkblue}}
\definecolor{fuchsia}{RGB}{255, 0, 255}
        \newcommand\fuchsia{\color{fuchsia}}
\newcommand\change[1]{{\begin{color}{mydarkblue}\sfbf #1\end{color}}}
\newcommand\mchange[1]{{\begin{color}{mydarkblue}\mathbf #1\end{color}}}
\newcommand\keep[1]{{\begin{color}{fuchsia}\sfbf #1\end{color}}}
%%%%%%%%%%%% 
\newcommand\james{M\hskip-.1ex\raise .575ex \hbox{\text{c}}\hskip-.075ex Kernan\xspace}
\newcommand\ifft{if and only if\xspace}
\newcommand\sfref[1]{{\sf\protect{\relsize{-.5}\ref{#1}}}}
\newcommand\stepref[1]{{\sf\protect{\relsize{-.5}\refblue Step~\ref{#1}}}}
\newcommand\demoref[1]{{\sf(\protect{\relsize{-.5}\ref{#1}})}}
\renewcommand\eqref{\demoref}
%\newcommand\james{M\hskip-.1ex\raise .525ex\hbox{\text{c}}Kernan} %
%{M\raise .575ex \hbox{\text{c}}Kernan}

%  
\renewcommand\thesubsection{\thesection.\Alph{subsection}}
\renewcommand\subsection{
  \renewcommand{\sfdefault}{phv}%{pag}
  \@startsection{subsection}%
  {2}{0pt}{-\baselineskip}{.2\baselineskip}{\raggedright
    \sffamily\itshape\relsize{-.5}%\small
  }}
\renewcommand\section{
  \renewcommand{\sfdefault}{phv}
  \@startsection{section} %
  {1}{0pt}{\baselineskip}{.2\baselineskip}{\centering
    \sffamily
    \scshape
    %\bfseries
}}

%%%%%%%%%%
\setlist[enumerate, 1]{itemsep=3pt,topsep=3pt,leftmargin=1.5em,font=\upshape,
  label={(\roman*)}}
%%%%%%%%%%
\newlist{enumfull}{enumerate}{1}
\setlist[enumfull]{itemsep=3pt,topsep=3pt,leftmargin=3.25em,font=\upshape,
  label={(\thethm.\arabic*\/)}}
%%%%%%%%%%\setlist[enumerate, 1]{font=\upshape}
%%%%%%%%%%
%%%%%%%%%%
\newlist{enumbold}{enumerate}{1}
\setlist[enumbold]{itemsep=3pt,topsep=3pt,leftmargin=1.95em,font=\upshape,
  label={%\relsize{-.5}
    \textbf{({\it\textbf{\roman*}}\/)}}}
%,itemindent=-1em} 
%\setlist[enumerate,1]{label=(\roman*)}
\newlist{enumalpha}{enumerate}{1}
\setlist[enumalpha]{itemsep=3pt,topsep=3pt,leftmargin=2em,font=\upshape,
  label=(\alph*\/)}
%%%%%%%%%%
\newlist{widemize}{itemize}{1}
\setlist[widemize]{itemsep=3pt,topsep=3pt,leftmargin=1em,label=$\bullet$}
%%%%%%%%%%
\newlist{widenumerate}{enumerate}{1}
\setlist[widenumerate]{itemsep=3pt,topsep=3pt,leftmargin=1.75em,label=(\roman*)}
\newlist{widenumalpha}{enumerate}{1}
\setlist[widenumalpha]{itemsep=3pt,topsep=3pt,leftmargin=1.5em,label=(\alph*)}
%%%%%%%%%%
\newcounter{parentthmnumber}
\setcounter{parentthmnumber}{0}
\newcounter{currentparentthmnumber}
\setcounter{currentparentthmnumber}{0}
\newcounter{nexttag}
\newcommand{\setnexttag}{%
  \setcounter{nexttag}{\value{enumi}}%
  \addtocounter{nexttag}{1}%
}
\newcommand{\placenexttag}{%
\tag{\roman{nexttag}}%
}

\newenvironment{thmlista}{%
\label{parentthma}
\begin{enumerate}%thmlistaa}
}{%
\end{enumerate}%thmlistaa}
}
\newlist{thmlistaa}{enumerate}{1}
\setlist[thmlistaa]{label=(\arabic*), ref=\autoref{parentthm}\thethm(\arabic*)}
%%%%%%%%%%
%%%%%%%%%%
\newcommand*{\parentthmlabeldef}{%
  \expandafter\newcommand
  \csname parentthm\the\value{parentthmnumber}\endcsname
}
\newcommand*{\ptlget}[1]{%
  \romannumeral-`\x
  \ltx@ifundefined{parentthm\number#1}{%
    \ltx@space
    \parentthmundefined
  }{%
    \expandafter\ltx@space
    \csname mymacro\number#1\endcsname
  }%
}  
\newcommand*{\parentthmundefined}{\textbf{??}}
\parentthmlabeldef{parentthm}
%%%%%%%%%%
%%%%%%%%%%
\newenvironment{thmlistr}{%
\label{parentthm}%{parentthmnumber}}
%\refstepcounter{parentthmnumber}
\begin{thmlistrr}}{%
\end{thmlistrr}}
\newlist{thmlistrr}{enumerate}{1}
\setlist[thmlistrr]{label=(\roman*), ref=\autoref{parentthm}(\roman*)}
%%%%%%%%%%
%%%%%%%%%%
%%%%%%%%%%
%%%%%%%%%%
%%%%%%%%%%
%\renewcommand\@cite[2]{{\rm [{#1\ifthenelse{\boolean{@tempswa}}{,\nolinebreak[3] #2}{}}]}}
%%%%%%%%%%
\newcounter{proofstep}%
\setcounter{proofstep}{0}%
\newcommand{\pstep}[1]{%
  \smallskip
  \noindent
  \emph{{\sc Step \arabic{proofstep}:} #1.}\addtocounter{proofstep}{1}}
%%%%%%%%%%
\newcounter{lastyear}\setcounter{lastyear}{\the\year}
\addtocounter{lastyear}{-1}
%%%%%%%%%%%
\newcommand\sideremark[1]{%
\normalmarginpar
\marginpar
[
\hskip .45in
\begin{minipage}{.75in}
\tiny #1
\end{minipage}
]
{
\hskip -.075in
\begin{minipage}{.75in}
\tiny #1
\end{minipage}
}}
\newcommand\rsideremark[1]{
\reversemarginpar
\marginpar
[
\hskip .45in
\begin{minipage}{.75in}
\tiny #1
\end{minipage}
]
{
\hskip -.075in
\begin{minipage}{.75in}
\tiny #1
\end{minipage}
}}
%%%%%%%%%
\newcommand\Index[1]{{#1}\index{#1}}
\newcommand\inddef[1]{\emph{#1}\index{#1}}
\newcommand\noin{\noindent}
\newcommand\hugeskip{\bigskip\bigskip\bigskip}
\newcommand\smc{\sc}
\newcommand\dsize{\displaystyle}
\newcommand\sh{\subheading}
\newcommand\nl{\newline}
%%%%%%%%%%%
%%%%%%%%% bibliography helpers
%%%%%%%%%%%
\newcommand\get{\input /home/kovacs/tex/latex/} %$ 
\newcommand\Get{\Input /home/kovacs/tex/latex/} %$ 
\newcommand\toappear{\rm (to appear)}
\newcommand\mycite[1]{[#1]}
\newcommand\myref[1]{(\ref{#1})}
\newcommand{\parref}[1]{\eqref{\bf #1}}
\newcommand\myli{\hfill\newline\smallskip\noindent{$\bullet$}\quad}
\newcommand\vol[1]{{\bf #1}\ } 
\newcommand\yr[1]{\rm (#1)\ } 
%%%%%%%%%%%
%%%%%%%%% text abbreviations
%%%%%%%%%%%
\newcommand\cf{cf.\ \cite}
\newcommand\mycf{cf.\ \mycite}
\newcommand\te{there exist\xspace}
\newcommand\st{such that\xspace}
\newcommand\CM{Cohen-Macaulay\xspace}
\newcommand\GR{Grauert-Riemenschneider\xspace}
\newcommand\notinclass{{\relsize{-.5}\sf [not discussed in class]}\xspace}
%%%%%%%%%%%
%%%%%%%%%%% Theorem Style: BOZONT
%%%%%%%%%%%
\newcommand\myskip{3pt}
\newtheoremstyle{bozont}{3pt}{3pt}%
     {\itshape}%         Body font
     {}%         Indent amount (empty = no indent, \parindent = para indent)
     {\bfseries}% Thm head font
     {.}%        Punctuation after thm head
     {.5em}%     Space after thm head (\newline = linebreak)
     {\thmname{#1}\thmnumber{ #2}\thmnote{\normalsize%\relsize{-.5}
        \ \rm #3}}
     % Thm head spec
%%%%%%%%%%%%%%%%%%%%%%%%%%%%%%
\newtheoremstyle{bozont-sub}{3pt}{3pt}%
     {\itshape}% Body font
     {}% Indent amount (empty = no indent, \parindent = para indent)
     {\bfseries}% Thm head font
     {.}% Punctuation after thm head
     {.5em}% Space after thm head (\newline = linebreak)
     {\thmname{#1}\ \arabic{section}.\arabic{thm}.\thmnumber{#2}\thmnote{\normalsize%\relsize{-.5}
 \ \rm #3}}
     % Thm head spec 
%%%%%%%%%%%%%%%%%%%%%%%%%%%%%%
\newtheoremstyle{bozont-named-thm}{3pt}{3pt}%
     {\itshape}%         Body font
     {}%         Indent amount (empty = no indent, \parindent = para indent)
     {\bfseries}% Thm head font
     {.}%        Punctuation after thm head
     {.5em}%     Space after thm head (\newline = linebreak)
     {\thmname{#1}\thmnumber{#2}\thmnote{ #3}}%         Thm head spec
%%%%%%%%%%%%%%%%%%%%%%%%%%%%%%
\newtheoremstyle{bozont-named-bf}{3pt}{3pt}%
     {}%         Body font
     {}%         Indent amount (empty = no indent, \parindent = para indent)
     {\bfseries}% Thm head font
     {.}%        Punctuation after thm head
     {.5em}%     Space after thm head (\newline = linebreak)
     {\thmname{#1}\thmnumber{#2}\thmnote{ #3}}%         Thm head spec
%%%%%%%%%%%%%%%%%%%%%%%%%%%%%%
\newtheoremstyle{bozont-named-sf}{3pt}{3pt}%
     {}%         Body font
     {}%         Indent amount (empty = no indent, \parindent = para indent)
     {\sffamily}% Thm head font
     {.}%        Punctuation after thm head
     {.5em}%     Space after thm head (\newline = linebreak)
     {\thmname{#1}\thmnumber{#2}\thmnote{ #3}}%         Thm head spec
%%%%%%%%%%%%%%%%%%%%%%%%%%%%%%
\newtheoremstyle{bozont-named-sc}{3pt}{3pt}%
     {}%         Body font
     {}%         Indent amount (empty = no indent, \parindent = para indent)
     {\scshape}% Thm head font
     {.}%        Punctuation after thm head
     {.5em}%     Space after thm head (\newline = linebreak)
     {\thmname{#1}\thmnumber{#2}\thmnote{ #3}}%         Thm head spec
%%%%%%%%%%%%%%%%%%%%%%%%%%%%%%
\newtheoremstyle{bozont-named-it}{3pt}{3pt}%
     {}%         Body font
     {}%         Indent amount (empty = no indent, \parindent = para indent)
     {\itshape}% Thm head font
     {.}%        Punctuation after thm head
     {.5em}%     Space after thm head (\newline = linebreak)
     {\thmname{#1}\thmnumber{#2}\thmnote{ #3}}%         Thm head spec
%%%%%%%%%%%%%%%%%%%%%%%%%%%%%%
\newtheoremstyle{bozont-sf}{3pt}{3pt}%
     {}%         Body font
     {}%         Indent amount (empty = no indent, \parindent = para indent)
     {\sffamily}% Thm head font
     {.}%        Punctuation after thm head
     {.5em}%     Space after thm head (\newline = linebreak)
     {\thmname{#1}\thmnumber{ #2}\thmnote{\normalsize%\relsize{-.5}
 \ \rm #3}}%         Thm head spec
%%%%%%%%%%%%%%%%%%%%%%%%%%%%%%
\newtheoremstyle{bozont-sc}{3pt}{3pt}%
     {}%         Body font
     {}%         Indent amount (empty = no indent, \parindent = para indent)
     {\scshape}% Thm head font
     {.}%        Punctuation after thm head
     {.5em}%     Space after thm head (\newline = linebreak)
     {\thmname{#1}\thmnumber{ #2}\thmnote{\normalsize%\relsize{-.5}
 \ \rm #3}}%         Thm head spec
%%%%%%%%%%%%%%%%%%%%%%%%%%%%%%
\newtheoremstyle{bozont-remark}{3pt}{3pt}%
     {}%         Body font
     {}%         Indent amount (empty = no indent, \parindent = para indent)
     {\scshape}% Thm head font
     {.}%        Punctuation after thm head
     {.5em}%     Space after thm head (\newline = linebreak)
     {\thmname{#1}\thmnumber{ #2}\thmnote{\normalsize%\relsize{-.5}
 \ \rm #3}}%         Thm head spec
%%%%%%%%%%%%%%%%%%%%%%%%%%%%%%
\newtheoremstyle{bozont-subremark}{3pt}{3pt}%
     {}%         Body font
     {}%         Indent amount (empty = no indent, \parindent = para indent)
     {\scshape}% Thm head font
     {.}%        Punctuation after thm head
     {.5em}%     Space after thm head (\newline = linebreak)
     {\thmname{#1}\ \arabic{section}.\arabic{thm}.\thmnumber{#2}\thmnote{\normalsize%\relsize{-.5}
 \ \rm #3}}
     % Thm head spec 
%%%%%%%%%%%%%%%%%%%%%%%%%%%%%%
\newtheoremstyle{bozont-def}{3pt}{3pt}%
     {}%         Body font
     {}%         Indent amount (empty = no indent, \parindent = para indent)
     {\bfseries}% Thm head font
     {.}%        Punctuation after thm head
     {.5em}%     Space after thm head (\newline = linebreak)
     {\thmname{#1}\thmnumber{ #2}\thmnote{\normalsize%\relsize{-.5}
 \ \rm #3}}%         Thm head spec
%%%%%%%%%%%%%%%%%%%%%%%%%%%%%%
\newtheoremstyle{bozont-reverse}{3pt}{3pt}%
     {\itshape}%         Body font
     {}%         Indent amount (empty = no indent, \parindent = para indent)
     {\bfseries}% Thm head font
     {.}%        Punctuation after thm head
     {.5em}%     Space after thm head (\newline = linebreak)
     {\thmnumber{#2.}\thmname{ #1}\thmnote{\normalsize%\relsize{-.5}
 \ \rm #3}}%         Thm head spec
%%%%%%%%%%%%%%%%%%%%%%%%%%%%%%
\newtheoremstyle{bozont-reverse-sc}{3pt}{3pt}%
     {\itshape}%         Body font
     {}%         Indent amount (empty = no indent, \parindent = para indent)
     {\scshape}% Thm head font
     {.}%        Punctuation after thm head
     {.5em}%     Space after thm head (\newline = linebreak)
     {\thmnumber{#2.}\thmname{ #1}\thmnote{\normalsize%\relsize{-.5}
 \ \rm #3}}%         Thm head spec
%%%%%%%%%%%%%%%%%%%%%%%%%%%%%%
\newtheoremstyle{bozont-reverse-sf}{3pt}{3pt}%
     {\itshape}%         Body font
     {}%         Indent amount (empty = no indent, \parindent = para indent)
     {\sffamily}% Thm head font
     {.}%        Punctuation after thm head
     {.5em}%     Space after thm head (\newline = linebreak)
     {\thmnumber{#2.}\thmname{ #1}\thmnote{\normalsize%\relsize{-.5}
 \ \rm #3}}%         Thm head spec
%%%%%%%%%%%%%%%%%%%%%%%%%%%%%%
\newtheoremstyle{bozont-remark-reverse}{3pt}{3pt}%
     {}%         Body font
     {}%         Indent amount (empty = no indent, \parindent = para indent)
     {\sc}% Thm head font
     {.}%        Punctuation after thm head
     {.5em}%     Space after thm head (\newline = linebreak)
     {\thmnumber{#2.}\thmname{ #1}\thmnote{\normalsize%\relsize{-.5}
 \ \rm #3}}%         Thm head spec
%%%%%%%%%%%%%%%%%%%%%%%%%%%%%%
\newtheoremstyle{bozont-def-reverse}{3pt}{3pt}%
     {}%         Body font
     {}%         Indent amount (empty = no indent, \parindent = para indent)
     {\sffamily}% Thm head font
     {.}%        Punctuation after thm head
     {.5em}%     Space after thm head (\newline = linebreak)
     {\thmnumber{{\relsize{-.5}(#2)}}\thmname{ {\relsize{-.25}#1}}\thmnote{{\relsize{-.5}\ \sf #3}}}%         Thm head spec
%%%%%%%%%%%%%%%%%%%%%%%%%%%%%%
\newtheoremstyle{bozont-def-newnum-reverse}{3pt}{3pt}%
     {}%         Body font
     {}%         Indent amount (empty = no indent, \parindent = para indent)
     {\bfseries}% Thm head font
     {}%        Punctuation after thm head
     {.5em}%     Space after thm head (\newline = linebreak)
     {\thmnumber{#2.}\thmname{ #1}\thmnote{\normalsize%\relsize{-.5}
 \ \rm #3}}%         Thm head spec
%%%%%%%%%%%%%%%%%%%%%%%%%%%%%%
\newtheoremstyle{bozont-def-newnum-reverse-plain}{3pt}{3pt}%
   {}% Body font
   {}% Indent amount (empty = no indent, \parindent = para indent)
   {}% Thm head font
   {}% Punctuation after thm head
   {.5em}% Space after thm head (\newline = linebreak)
   {\thmnumber{\!(#2)}\thmname{ #1}\thmnote{\normalsize%\relsize{-.5}
 \ \rm #3}}% Thm head spec
%%%%%%%%%%%%%%%%%%%%%%%%%%%%%%
\newtheoremstyle{bozont-number}{3pt}{3pt}%
   {}% Body font
   {}% Indent amount (empty = no indent, \parindent = para indent)
   {}% Thm head font
   {}% Punctuation after thm head
   {0pt}% Space after thm head (\newline = linebreak)
   {\thmnumber{\!(#2)} }% Thm head spec
%%%%%%%%%%%%%%%%%%%%%%%%%%%%%%
\newtheoremstyle{bozont-say}{3pt}{3pt}%
     {}%         Body font
     {}%         Indent amount (empty = no indent, \parindent = para indent)
     {\sffamily}% Thm head font
     {.}%        Punctuation after thm head
     {.5em}%     Space after thm head (\newline = linebreak)
     {\thmnumber{{\relsize{-.5}\S#2}}\thmname{ {\relsize{-.25}#1}}%
       \ \thmnote{%{\relsize{-.5}\
         \it #3}}%}         Thm head spec
%%%%%%%%%%%%%%%%%%%%%%%%%%%%%%
\newtheoremstyle{bozont-subsay}{3pt}{3pt}%
     {}%         Body font
     {}%         Indent amount (empty = no indent, \parindent = para indent)
     {\sffamily}% Thm head font
     {.}%        Punctuation after thm head
     {.5em}%     Space after thm head (\newline = linebreak)
     {\thmnumber{{\relsize{-.5}\S\S#2}}\thmname{ {\relsize{-.25}#1}}\thmnote{{\relsize{-.5}\ \sf #3}}}%         Thm head spec
%%%%%%%%%%%%%%%%%%%%%%%%%%%%%%
\newtheoremstyle{bozont-step}{3pt}{3pt}%
   {\itshape}% Body font
   {}% Indent amount (empty = no indent, \parindent = para indent)
   {\scshape}% Thm head font
   {}% Punctuation after thm head
   {.5em}% Space after thm head (\newline = linebreak)
   {$\boxed{\text{\sc \thmname{#1}~\thmnumber{#2}:\!}}$}
   % Thm head spec 
%%%%%%%%%%%%%%%%%%%%%%%%%%%%%%
%%%%%%%%%%%%%%%%%%%%%%%%%%%%%%
%%%%%%%%%%%%%%%%%%%%%%%%%%%%%% thms
%%%%%%%%%%%%%%%%%%%%%%%%%%%%%%
\theoremstyle{bozont}    
\ifnum \value{are-there-sections}=0 {%
  \basetheorem{proclaim}{Theorem}{}
} 
\else {%
  \basetheorem{proclaim}{Theorem}{section}
} 
\fi
%%%%%%%%%%%%%%%%%%%%%%%%%%%%%%
\maketheorem{thm}{proclaim}{Theorem}
\maketheorem{mainthm}{proclaim}{Main Theorem}
\maketheorem{cor}{proclaim}{Corollary} 
\maketheorem{cors}{proclaim}{Corollaries} 
\maketheorem{lem}{proclaim}{Lemma} 
\maketheorem{prop}{proclaim}{Proposition} 
\maketheorem{conj}{proclaim}{Conjecture}
\basetheorem{subproclaim}{Theorem}{proclaim}
\maketheorem{sublemma}{subproclaim}{Lemma}
\newenvironment{sublem}{%
\setcounter{sublemma}{\value{equation}}
\begin{sublemma}}
{\end{sublemma}}
%%%%%%%%%%%%%%%%%%%%%%%%%%%%%%
\theoremstyle{bozont-sub}
%\maketheorem{subproclaim}{equation}{Theorem}
\maketheorem{subthm}{equation}{Theorem}
\maketheorem{subcor}{equation}{Corollary} 
%\maketheorem{sublem}{equation}{Lemma} 
\maketheorem{subprop}{equation}{Proposition} 
\maketheorem{subconj}{equation}{Conjecture}
%%%%%%%%%%%%%%%%%%%%%%%%%%%%%%
\theoremstyle{bozont-named-thm}
\maketheorem{namedthm}{proclaim}{}
%%%%
\theoremstyle{bozont-sc}
\newtheorem{proclaim-special}[proclaim]{\specialthmname}
\newenvironment{proclaimspecial}[1]
     {\def\specialthmname{#1}\begin{proclaim-special}}
     {\end{proclaim-special}}
%%%%%%%%%%%%%%%%%%%%%%%%%%%%%%
\theoremstyle{bozont-subremark}
\basetheorem{subremark}{Remark}{proclaim}
\maketheorem{subrem}{equation}{Remark}
\maketheorem{subnotation}{equation}{Notation} 
\maketheorem{subassume}{equation}{Assumptions} 
\maketheorem{subobs}{equation}{Observation} 
\maketheorem{subexample}{equation}{Example} 
\maketheorem{subex}{equation}{Exercise} 
\maketheorem{inclaim}{equation}{Claim} 
%\maketheorem{subclaim}{equation}{Claim} 
\maketheorem{subquestion}{equation}{Question}
%%%%%%%%%%%%%%%%%%%%%%%%%%%%%%
\theoremstyle{bozont-remark}
\basetheorem{remark}{Remark}{proclaim}
\makeremark{subclaim}{subremark}{Claim}
\maketheorem{rem}{proclaim}{Remark}
\maketheorem{rems}{proclaim}{Remarks}   %%JK 0327

\maketheorem{claim}{proclaim}{Claim} 
\maketheorem{notation}{proclaim}{Notation} 
\maketheorem{assume}{proclaim}{Assumptions} 
\maketheorem{assumeone}{proclaim}{Assumption} 
\maketheorem{obs}{proclaim}{Observation} 
\maketheorem{example}{proclaim}{Example} 
\maketheorem{examples}{proclaim}{Examples} 
\maketheorem{complem}{equation}{Complement}%!!!!!!!!!!!!!!!!!!!!!!
\maketheorem{const}{proclaim}{Construction}   %!!!!!!!!!!!!!!!!
\maketheorem{ex}{proclaim}{Exercise} 
\newtheorem{case}{Case} 
\newtheorem{subcase}{Subcase}   
\newtheorem{step}{Step}
\newtheorem{approach}{Approach}
\maketheorem{Fact}{proclaim}{Fact}
\newtheorem{fact}{Fact}
%%%%%%%%%%%%%%%%%%%%%%%%%%%%%%%%%%%%%%
\newtheorem*{SubHeading*}{\SubHeadingName}%
\newtheorem{SubHeading}[proclaim]{\SubHeadingName}
\newtheorem{sSubHeading}[equation]{\sSubHeadingName}
\newenvironment{demo}[1] {\def\SubHeadingName{#1}\begin{SubHeading}}
  {\end{SubHeading}}%  
\newenvironment{subdemo}[1]{\def\sSubHeadingName{#1}\begin{sSubHeading}}
  {\end{sSubHeading}} %
\newenvironment{demor}[1]{\def\SubHeadingName{#1}\begin{SubHeading-r}}
  {\end{SubHeading-r}}%
\newenvironment{subdemor}[1]{\def\sSubHeadingName{#1}\begin{sSubHeading-r}}
  {\end{sSubHeading-r}} %
\newenvironment{demo-r}[1]{%
  \def\SubHeadingName{#1}\begin{SubHeading-r}}
  {\end{SubHeading-r}}%
\newenvironment{subdemo-r}[1]{\def\sSubHeadingName{#1}\begin{sSubHeading-r}}
  {\end{sSubHeading-r}} %
\newenvironment{demo*}[1]{\def\SubHeadingName{#1}\begin{SubHeading*}}
  {\end{SubHeading*}}%
%%%%%%%%%%%%%%%%%%%%%%%%%%%%%%%%%%%%%%
\maketheorem{defini}{proclaim}{Definition}
\maketheorem{defnot}{proclaim}{Definitions and notation}
\maketheorem{question}{proclaim}{Question}
\maketheorem{terminology}{proclaim}{Terminology}
\maketheorem{crit}{proclaim}{Criterion}
\maketheorem{pitfall}{proclaim}{Pitfall}%!!!!!!!!!!!!!!!!!!!!!!
\maketheorem{addition}{proclaim}{Addition}%!!!!!!!!!!!!!!!!!!!!!!
\maketheorem{principle}{proclaim}{Principle} %!!!!!!!!!!!!!!!!!!!!!!
%%%%%%%%%%%%%%%%%%%%%%%%%%%%%%
%%%  these were at once \theoremstyle{bozont-def}
\maketheorem{condition}{proclaim}{Condition}
\maketheorem{exmp}{proclaim}{Example}
\maketheorem{hint}{proclaim}{Hint}
\maketheorem{exrc}{proclaim}{Exercise}
\maketheorem{prob}{proclaim}{Problem}
\maketheorem{ques}{proclaim}{Question}    %!!!!!!!!!!!!!!!!!!!!
\maketheorem{alg}{proclaim}{Algorithm}
\maketheorem{remk}{proclaim}{Remark}          
\maketheorem{note}{proclaim}{Note}            
\maketheorem{summ}{proclaim}{Summary}         
\maketheorem{notationk}{proclaim}{Notation}   
\maketheorem{warning}{proclaim}{Warning}  
\maketheorem{defn-thm}{proclaim}{Definition--Theorem}  %!!!!!!!!!!!!!!!!!!!!!!!!
\maketheorem{convention}{proclaim}{Convention}  %!!!!!!!!!!!!!!!!!!!!!!!!!!!
\maketheorem{hw}{proclaim}{Homework}
\maketheorem{hws}{proclaim}{\protect{${\mathbb\star}$}Homework}
%%%%%%%%%%%%%%%%%%%%%%%%%%%%%%
\newtheorem*{ack}{Acknowledgment}
\newtheorem*{acks}{Acknowledgments}
%%%%%%%%%%%%%%%%%%%%%%%%%%%%%%
\theoremstyle{bozont-number}
\theoremstyle{bozont-def}    
\maketheorem{defn}{proclaim}{Definition}
\maketheorem{subdefn}{equation}{Definition}
%%%%%%%%%%%%%%%%%%
\theoremstyle{bozont-reverse}    
% \ifnum \value{are-there-sections}=0 {%
% %  \newtheorem{proclaim}{Theorem}{}
%   \basetheorem{proclaimr}{Theorem}{}
% } 
% \else {%
% %  \newtheorem{proclaim}{Theorem}[section]
%   \basetheorem{proclaimr}{Theorem}{section}
% } 
% \fi
%%%%%%%%%%%%%%%%%%%%%%%%%%%%%%
%\newtheorem{proclaim}{Theorem} [section]
\maketheorem{corr}{proclaim}{Corollary} 
\maketheorem{lemr}{proclaim}{Lemma} 
\maketheorem{propr}{proclaim}{Proposition} 
\maketheorem{conjr}{proclaim}{Conjecture}
%%%%
\theoremstyle{bozont-remark-reverse}
\newtheorem{SubHeading-r}[proclaim]{\SubHeadingName}
\newtheorem{sSubHeading-r}[equation]{\sSubHeadingName}
\newtheorem{SubHeadingr}[proclaim]{\SubHeadingName}
\newtheorem{sSubHeadingr}[equation]{\sSubHeadingName}
%%%%
\theoremstyle{bozont-reverse-sc}
\newtheorem{proclaimr-special}[proclaim]{\specialthmname}
\newenvironment{proclaimspecialr}[1]%
{\def\specialthmname{#1}\begin{proclaimr-special}}%
{\end{proclaimr-special}}
%%%%%%%%%%%%%%%%%%%%%%%%%%%%%%
\theoremstyle{bozont-remark-reverse}
\maketheorem{remr}{proclaim}{Remark}
\maketheorem{subremr}{equation}{Remark}
\maketheorem{notationr}{proclaim}{Notation} 
\maketheorem{assumer}{proclaim}{Assumptions} 
\maketheorem{obsr}{proclaim}{Observation} 
\maketheorem{exampler}{proclaim}{Example} 
\maketheorem{exr}{proclaim}{Exercise} 
\maketheorem{claimr}{proclaim}{Claim} 
\maketheorem{inclaimr}{equation}{Claim} 
\maketheorem{definir}{proclaim}{Definition}
%%%%%%%%%%%%%%%%%%%%%%%%%%%%%%
\theoremstyle{bozont-def-newnum-reverse}    
\maketheorem{newnumr}{proclaim}{}
\theoremstyle{bozont-def-newnum-reverse-plain}
\maketheorem{newnumrp}{proclaim}{}
%%%%%%%%%%%%%%%%%%%%%%%%%%%%%%
\theoremstyle{bozont-def-reverse}    
\maketheorem{defnr}{proclaim}{Definition}
\maketheorem{questionr}{proclaim}{Question}
\newtheorem{newnumspecial}[proclaim]{\specialnewnumname}
\newenvironment{newnum}[1]{\def\specialnewnumname{#1}\begin{newnumspecial}}{\end{newnumspecial}}
%%%%%%%%%%%%%%%%%%%%%%%%%%%%%%
\theoremstyle{bozont-say}    
\newtheorem{say}[proclaim]{}
\theoremstyle{bozont-subsay}    
\newtheorem{subsay}[equation]{}
%%%%%%%%%%%%%%%%%%%%%%%%%%%%%%
\theoremstyle{bozont-step}
\newtheorem{bstep}{Step}
%%%%%%%%%%%%%%%%%%%%%%%%%%%%%%%%%%%%%%%%%%%%%%%%%%%%%%%%%%%%%%%%%%
%%%%%%%%%%%%%%%%  Labels and Refs for Items %%%%%%%%%%%%%%%%%%%%%%
\newcounter{thisthm} 
\newcounter{thissection} 
%%%%%%%%%%%%%%%% WARNING !!!!!!!!!! %%%%%%%%%%%%%%%%
%%%%%%%%%%%%%%%% \ilabel -- \iref only works within the same thm !!!!!! %%%%%%%%%%%%%%%%
\newcommand{\ilabel}[1]{%
  \newcounter{#1}%
  \setcounter{thissection}{\value{section}}%
  \setcounter{thisthm}{\value{proclaim}}%
  \label{#1}}
\newcommand{\iref}[1]{%
  (\the\value{thissection}.\the\value{thisthm}.\ref{#1})}
%%%%%%%%%%%%%%%%%%%%%%%%%%%%%%%%%%%%%%%%%%%%%%%%%%%%%%%%%%%%%%%%%%
\newcounter{lect}
\setcounter{lect}{1}
\newcommand\resetlect{\setcounter{lect}{1}\setcounter{page}{0}}
\newcommand\lecture{\newpage\centerline{\sfbf Lecture \arabic{lect}}
  \addtocounter{lect}{1}}
\newcommand\nnplecture{\hugeskip\centerline{\sfbf Lecture \arabic{lect}}
\addtocounter{lect}{1}}
\newcounter{topic}
\setcounter{topic}{1}
\newenvironment{topic}
{\noindent{\sc Topic %Lecture 
\arabic{topic}:\ }}{\addtocounter{topic}{1}\par}
%%%%%%%%%%%%%%%%%%%%%%%%%%%%%%%%%%%%%%%%%%%%%%%%%%%%%%%%%%%%%%%%%%
%%%%%%%%%%%%%%%%%%
\counterwithin{equation}{proclaim}
\counterwithin{enumfulli}{equation}
\counterwithin{figure}{section} 
\newcommand\equinsect{\numberwithin{equation}{section}}
\newcommand\equinthm{\numberwithin{equation}{proclaim}}
\newcommand\figinthm{\numberwithin{figure}{proclaim}}
\newcommand\figinsect{\numberwithin{figure}{section}}
%% The next environment produces equations that are numbered within the section, not the
%% theorem. It also increases the counter for thm. This is to be used at the
%% beginning of a section to avoid reference numbers containing 0. It may also be
%% used for equations that are not part of a numbered statement. Otherwise equations
%% take the last numbered environment's number and this is not always desirable.
\newenvironment{sequation}{%
\setcounter{equation}{\value{thm}}
\numberwithin{equation}{section}%
\begin{equation}%
}{%
\end{equation}%
\numberwithin{equation}{proclaim}%
\addtocounter{proclaim}{1}%
}
%%%%%%%%%%%%%%%
\newcommand{\num}{\arabic{section}.\arabic{proclaim}}
%%%%%%%%%%%%%%%%%
\newenvironment{pf}{\smallskip \noindent {\sc Proof. }}{\qed\smallskip}
%%%%%%%%%%%%%%%%%
\newenvironment{enumerate-p}{
  \begin{enumerate}}
  {\setcounter{equation}{\value{enumi}}\end{enumerate}}
\newenvironment{enumerate-cont}{
  \begin{enumerate}
    {\setcounter{enumi}{\value{equation}}}}
  {\setcounter{equation}{\value{enumi}}
  \end{enumerate}}
%%%%%%%%%%%
\let\lenumi\labelenumi
\newcommand{\rmlabels}{\renewcommand{\labelenumi}{\rm \lenumi}}
\newcommand{\rmlabelsoff}{\renewcommand{\labelenumi}{\lenumi}}
%%%%%%%%%%%
\newenvironment{heading}{\begin{center} \sc}{\end{center}}
%%%%%%%%%%%
\newcommand\subheading[1]{\smallskip\noindent{{\bf #1.}\ }}
%%%%%%%%%%%%%%%%%%%%%%%%%%%%%%
\newlength{\swidth}
\setlength{\swidth}{\textwidth}
\addtolength{\swidth}{-,5\parindent}
\newenvironment{narrow}{
  \medskip\noindent\hfill\begin{minipage}{\swidth}}
  {\end{minipage}\medskip}
%%%%%%%%%%%%%%%%%%%%%%%%%%%%%%
\newcommand\nospace{\hskip-.45ex}
\newcommand{\sfbf}{\sffamily\bfseries}
\newcommand{\sfbfs}{\sffamily\bfseries\relsize{-.5}
}
\newcommand{\twidle}{\textasciitilde}
\makeatother

\newcounter{stepp}
\setcounter{stepp}{0}
\newcommand{\nextstep}[1]{%
  \addtocounter{stepp}{1}%
%        \par\noindent $\boxed{\text{\sc Step \arabic{stepp}:\!}}$ %
  \begin{bstep}%
    {#1}%\newline\noindent
  \end{bstep}%
  \noindent%
}
\newcommand{\resetsteps}{\setcounter{stepp}{0}}
%%%%
\newcommand{\spec}{\Spec}
%%%%%%%%%%%%%%%%%%%%%%%%%
\title[KSB stability in codimension $\geq 3$]%
{KSB stability is automatic in codimension $\geq 3$} %long title
\author{J\'anos Koll\'ar} \date{\usdate\today}
\thanks{J\'anos Koll\'ar was supported in part by NSF Grants DMS-1901855. \\
  \indent S\'andor Kov\'acs was supported in part by NSF Grant DMS-2100389 and a
  Simons Fellowship.}

\email{kollar@math.princeton.edu}
\address{\vskip-1.75em Department of Mathematics, Princeton University, Fine
  Hall, Washington Road, Princeton, NJ 08544-1000, USA} 
%\urladdr{http://www.math.princeton.edu/$\sim$kollar\xspace}
\author{S\'andor J Kov\'acs}
\email{skovacs@uw.edu\xspace}
\address{\vskip-1.75em University of Washington, Department of Mathematics, Box 354350,
Seattle, WA 98195-4350, USA}
% \keywords{}
% \subjclass[2000]{14}
%
\maketitle
\setcounter{tocdepth}{1}
%\tableofcontents
%%%%%%%%%%%%%%%%%%%%%%%
\numberwithin{equation}{thm}

\newcommand{\toploccohs}{liftable local cohomology\xspace}
\newcommand{\toploccohsi}{liftable $i^\text{th}\!$ local cohomology\xspace}
\newcommand{\toploccohst}{liftable $i^\text{th}\!$ local cohomology for $i\geq
  n-t$\xspace} 
\newcommand{\toploccohstm}{liftable local cohomology in codimension $t-1$\xspace}
%\newcommand{\Fp}{$F$-pure\xspace}
%\newcommand{\Fan}{$F$-anti-nilpotent\xspace}

%\input{extras}
%\section{J\'anos}

\renewcommand{\o}{\sO}
\newcommand{\Z}{\bZ}
\newcommand{\C}{\bC}

\section{Introduction}

\noin The right framework for a moduli theory of canonical models of varieties of
general type was established in \cite{KSB88}, at least in characteristic $0$ and over Noetherian bases; both of which
we assume from now on. The resulting notion, now called {\it KSB stability,} works
with finitely presented, flat morphisms $g:X\to B$ that satisfy 3 requirements.
\begin{widemize}
\item (Global condition) $\omega_{X/B}$ is relatively ample and ${g}$ is projective,
\item (Fiberwise condition)  the fibers $X_b$ are semi-log-canonical, and
\item (Local stability condition) $\omega_{X/B}^{[m]}$ is flat over $B$ and commutes
  with base changes $B'\to B$ for every $m\in \bZ$.
\end{widemize}
If $g$ satisfies the last two, then it is called \emph{locally KSB
  stable}.
%% In practice, local stability is the most subtle one.
See \cite{ModBook} for a detailed discussion of the resulting moduli theory, especially
\cite[Sec.~6.2]{ModBook}.

Note that the local stability condition is automatic at codimension $1$ points, and
quite well understood at codimension $2$ points, since we have a complete
classification of $2$-dimensional slc singularities; see \cite{KSB88} and
\cite[Sec.~2.2]{ModBook}.
Our aim is to show that local stability is automatic in codimension $\geq 3$.
The simplest version is the following.

\begin{thm}\label{codim.3.simple.thm}
  Let $g:X\to B$ be a flat morphism of finite type over a field of characteristic 0.
  Let $Z\subset X$ be a closed subset such that $\codim (Z_b\subset X_b)\geq 3$ for
  every $b\in B$, and set $U:=X\setminus Z$.

  Assume that the fibers  $X_b$ are semi-log-canonical, and
  $g\resto U: U\to B$ is locally KSB stable. Then $g:X\to B$ is locally KSB stable.
\end{thm}

If the fibers $X_b$ are CM, the claim follows from \cite[10.73]{ModBook}.  Being CM
is a deformation invariant property for projective, locally stable families by
\cite{KK10}, see also \cite[2.67]{ModBook}.  In particular, the theorem was known to
hold for varieties in those connected components of the KSB moduli space that contain
a canonical model of a smooth variety.
    
If $B$ is reduced, the theorem is proved in \cite{MR3115196}, see
also \cite[5.6]{ModBook}. Thus it remains to deal with the case when $B=\Spec A$ for
an Artinian ring $A$, which implies the theorem for any $B$.

%% That
%% is, if $g:X\to B$ satisfies the fiberwise condition everywhere, and the local
%% stability condition at all points of codimension $\leq 2$, then it satisfies
%% the local stability condition everywhere.

For applications, and even for the proof of \autoref{codim.3.simple.thm}, we need a
form that strengthens it in 2 significant ways.  First, we deal with pairs
$(X,\Delta=\sum a_iD_i)$, where
$a_i\in \left\{\frac12, \frac23, \frac34, \dots, 1\right\}$ for every $i$; these are
frequently called {\it standard} coefficients.  Second, and this is more important,
we assume $g$ to be flat only in codimension $\leq 2$.

%% We say that a
%% family of such pairs $g:(X,\Delta=\sum a_iD_i)\to B$ is  \emph{locally KSB stable} if
%% \begin{widemize}
%% \item (Fiberwise condition) the fibers $(X_b,\Delta_b)$ are semi-log-canonical,
%%   and
%% \item (Local stability condition)
%%   $\omega_{X/B}^{[m]}\bigl(\tsum_i \rdown{ma_i}D_i\bigr)$ is flat over $B$ and
%%   commutes with base changes for every $m\in \bZ$.
%% \end{widemize}
%% If we exclude
%% $a_i=\frac12$, this is the same as
%% in \cite[Sec.~6.2]{ModBook}. We comment on the  $a_i=\frac12$ cases in
%% \autoref{codim.3.artin.thm.refs}.

%% Our main result is the following:

\begin{thm}\label{codim.3.artin.thm}
  Let $g:X\to B$ be a morphism of finite type and of pure relative dimension over a
  field of characteristic 0, and $\Delta=\sum a_iD_i$, where the $D_i$ are relative
  Mumford $\bZ$-divisors.  Let $Z\subset X$ be a closed subset and set
  $U:=X\setminus Z$.  Assume that
  %% Let $g:X\to B$ be a morphism of finite type of pure relative dimension over a
  %% field of characteristic 0, and $Z\subset X$ a closed subset and set
  %% $U\leteq X\setminus Z$.  Assume that
  \begin{enumfull}%[label=(\ref{codim.3.artin.thm}.\arabic*)]
  \item\label{item:23} $a_i\in \left\{\frac12, \frac23, \frac34, \dots, 1\right\}$
    for every $i$,
  \item\label{item:13} $\codim (Z_b\subset X_b)\geq 3$ for every $b\in B$,
  \item\label{item:14} $g\resto U: U\to B$ is flat and the  fibers
    $(U_b, \Delta|_{U_b})$ are semi-log-canonical,
  \item\label{item:14b}$\omega_{U/B}^{[m]}\bigl(\tsum_i \rdown{ma_i}D_i|_U\bigr)$ is
    flat over $B$ and commutes with base changes for every $m\in \bZ$,
    %% $g\resto U:(U, \Delta|_U)\to B$ is locally KSB stable,
  \item\label{item:15} $\depth_ZX\geq 2$, and
    % \rsideremark{\tiny S, 03/13/23: should we point out that this implies that $X$
    % is $S_2$?}%
  \item\label{item:16} the normalization $(\ol X_b, \ol C_b+\ol\Delta_b)\to X_b$ is
    log canonical for every $b\in B$, where $\ol C_b$ denotes the conductor of the
    normalization $\ol X_b\to X_b$; see \cite[5.2]{SingBook}.
    %%% the normalization $(\ol X_b, \ol D_b)\to X_b$ is log canonical for every
    %%% $b\in B$.
  \end{enumfull}
  Then
  \begin{enumfull}[resume]%\setcounter{enumfulli}{6}
  \item\label{item:14con} $g: X\to B$ is flat,
  \item\label{item:14acon} the fibers $(X_b, \Delta_b)$ are semi-log-canonical, and
  \item\label{item:14bcon} $\omega_{X/B}^{[m]}\bigl(\tsum_i \rdown{ma_i}D_i\bigr)$ is
    flat over $B$ and commutes with base changes for every $m\in \bZ$.
  \end{enumfull}
\end{thm}

\begin{rems}\label{codim.3.artin.thm.refs}\
  \begin{enumfull}
  \item As in \cite[4.68]{ModBook}, $D$ is a \emph{relative Mumford divisor} if
    at every generic point of $X_b\cap D$, the fiber $X_b$ is smooth and $D$ is
    Cartier.
      %% for more about Mumford divisors see \cite[Sec.~4.8]{ModBook}.
  \item The notation $\depth_ZX$ stands for
    $\depth_Z\sO_X\leteq\inf\{\depth_z\sO_X\skvert z\in Z\}$. This terminology is
    used, for instance, in \cite[(5.10.1)]{EGAIV2} and \cite[10.3]{ModBook}. %In fact,
  \item The condition \autoref{item:15} is easy to ensure
    % \autoref{codim.3.artin.thm}
    by replacing $\sO_X$ with the push-forward of $\sO_{U}$ if necessary.  If $B$ is
    $S_2$ then \autoref{item:15} holds iff $X$ is $S_2$.
  \item Assumption %\autoref{codim.3.artin.thm}
    \autoref{item:16} is a weakening of the fiberwise condition; the two are
    equivalent iff $X_b$ is $S_2$.
    In many applications, including the proof of
    \autoref{codim.3.artin.thm},
    % (\sfref{codim.3.artin.thm.pf}),
    at the beginning we only know
    % \autoref{codim.3.artin.thm}
    \autoref{item:16}, but eventually conclude that $(X_b,\Delta_b)$ is slc.
    % , in particular $X_b$ is $S_2$.
  \item\label{item:1.3.4.jk} The following may be a better way of formulating
    \autoref{item:16}.  Let $j:U\into X$ be the natural embedding and set
    $\wt X_b:=\spec_{X_b}j_*\sO_{U_b}$, which is the {\it demi-normalization} and
    also the {\it$S_2$-hull} of the fiber $X_b$; see \cite[Sec.5.1]{SingBook} and
    \cite[Sec.9.1]{ModBook}.  Then $\wt X_b\to X_b$ is a universal homeomorphism that
    is an isomorphism over $U_b$. Now \autoref{item:16} holds iff  the induced pair
    $\bigl(\wt X_b, \wt \Delta_b\bigr)$ is slc.
  \item If $a_i\in \left\{\frac23, \frac34, \dots\right\}$, then \autoref{item:14b} is the same as the main assumption of KSB
    stabilty with standard coefficients as defined in \cite[6.21.3]{ModBook}.
    %%\cite[Sec.~6.2]{ModBook}.

    If we allow $a_i=\frac12$, then the above definition treats the pairs $(X, D)$,
    $(X, \frac12 D+\frac12 D)$ and $(X, \frac12 (2D))$ as different objects.  Note
    that $\omega_X(\sum\rdown{a_i}D_i)$ is $\omega_X(D)$ in the first case but
    $\omega_X$ in the other 2 cases.
Thus,    replacing  $1\cdot D_i$ with $\frac12 D_i+\frac12 D_i$
ensures the extra condition on the $\{D_i\colon a_i=1\}$ in  \cite[6.22.3]{ModBook}.

    This way of handling the coefficient $\frac12$ case may not be natural from the point of view of
    moduli, but seems necessary;
    %% once a choice has been made, \autoref{codim.3.artin.thm} works as  stated
    see \cite[Secs.~8.1--2]{ModBook} for a discussion of the general notion
    of such \emph{marked pairs.}
\item The definition of  KSB
  stabilty with standard coefficients also requires the $D_i$ to be flat
  by \cite[6.21.1]{ModBook}. We do not know whether this is automatic in codimensions $\geq 3$, see \autoref{4.3.cor} for a special case.

    %%Once we assume this to hold over $U$, the proof gives that it also holds over $X$. 
    \item  We comment on other versions of stability in \autoref{sec.ksba.jk}.
    %% as in \cite[Sec.~8.1]{ModBook}. Accordingly, $\rdown{m\Delta}$ is defined as
    %% $\tsum_i \rdown{ma_i}D_i$, which matters if the $D_i$ are not pairwise
    %% distinct.For details about \emph{KSB stability} of marked stable pairs see
    %% \cite[Sec.~8.2]{ModBook}.
    %%
    %% Several special cases were proved earlier.
    %%
    %% \item The previous results do not imply \autoref{codim.3.artin.thm} for all
    %%   stable varieties, much less the local version as stated.
  
  \end{enumfull}
\end{rems}

\begin{say}[Sketch of an approach to \autoref{codim.3.artin.thm}]
  \label{codim.3.artin.thm.sketch}
  Assume for simplicity that we are over $\bC$, $B=\Spec A$ for an Artinian ring $A$,
  and the closed fiber $X_k$ is projective.  As in \cite{KK10} the proof relies on
  the Du~Bois property (see \autoref{rem:DB}) of slc varieties, which implies that
  the natural maps
  \begin{equation}
    \label{eq:33}\qquad\qquad
    H^i(X_k^{\rm an}, \bC)\onto H^i(X_k^{\rm an}, \sO_{X_k^{\rm an}})
    \qquad\qtq{are surjective.}
  \end{equation}
  If $g$ is also flat, these imply that the $R^ig_*\sO_X$ are (locally) free by
  \cite{MR0376678}.  Using this for various cyclic covers, \cite[2.68]{ModBook} shows
  that $\omega_{X/B}$ is flat over $B$ and commutes with base changes $B'\to B$.
%
  %% \sideremark{%\vskip -15em
  %% \tiny Doesn't this %(the flat implies commuting)
  %% need that $g$ be flat? (I can prove this with assuming that, but I don't know
  %% how
  %% to do it without $g$ being flat in this generality.)}%
  %% (Actually, for Artinian bases, flat implies commuting with base changes
  %% $B'\to B$, see ???.)

  An inspection of these proofs shows that, in order to get the flatness of
  $\omega_{X/B}$, we need \autoref{eq:33} only for $i=n,n{-}1$ where $n:=\dim
  X_k$. This is where the codimension $3$ condition enters first.
  As we noted in \autoref{item:1.3.4.jk}, the 
  %%  {codim.3.artin.thm.refs}
  {\it
    demi-normalization} $\wt X_k$ of $X_k$ is slc, and
  $\wt X_k\to X_k$  is a universal homeomorphism that is 
 an isomorphism over
 $U_k$.  Thus
 \[
 \begin{array}{llll}
 H^i(X_k^{\rm an}, \bC)&\simeq & H^i(\wt X_k^{\rm an}, \bC) &
 \mbox{for every $i$, and}\\
 H^i(X_k^{\rm an}, \sO_{X_k^{\rm an}}) &\simeq & H^i(\wt X_k^{\rm an}, \sO_{\wt X_k^{\rm an}})
 &
 \mbox{for $i=n,n{-}1$.}
 \end{array}
 \]
 %% Assumption \autoref{item:16} guarantees that $\wt X_k$ is slc;
 %%  cf.~\cite[5.10]{SingBook} and \autoref{item:13} implies that $X_k$ and its
 %%  deminormalization $\wt X_k$ have the same cohomologies in degrees $i=n,n{-}1$.
 It follows that \autoref{eq:33} holds for $i=n,n{-}1$, although $X_k$ is not (yet
 known to be) Du~Bois, see also \autoref{thm:DB-to-llc}.
 One also sees that it is enough if $g$ is flat at points of dimension $\geq
 n-2$.
 Therefore we get that $\omega_{X/B}$ is flat over $B$.
\end{say}
  
 Interestingly, this approach does not seem to imply that $X$ is flat over $B$, much
 less the full \autoref{codim.3.artin.thm}.  A possible explanation is that
 $\omega_X$ is insensitive to codimension 2:
 \begin{lem}\label{lem.1.1.jk}
   Let $\pi:Y\to X$ be a quasi-finite morphism that is an isomorphism at points of
   codimension $\leq 1$.  Then $\pi_*\omega_Y\simeq \omega_X$.
 \end{lem}
 \begin{proof}
   Let $\imath: U\into X$ be the largest open subset such that
   $\pi'\leteq\pi\resto {\pi^{-1}U}$ is an isomorphism between $\pi^{-1}U$ and $U$.
   Let $\jmath: \pi^{-1}U\into Y$ denote the embedding.  By assumption
   $\codim(Y\setminus \pi^{-1}U,Y)\geq 2$ and $\codim(X\setminus U,X)\geq
   2$. Therefore, because $\omega_X$ and $\omega_Y$ are $S_2$-sheaves
   (cf.~\cite[5.69]{KM98}), it follows that
     \[
       \pi_*\omega_Y\simeq \pi_*\jmath_*\omega_{\pi^{-1}U} \simeq
       \imath_*\pi'_*\omega_{\pi^{-1}U} \simeq \imath_*\omega_U \simeq
       \omega_X. \qedhere
     \]
 \end{proof}

 % That is, if $\pi:Y\to X$ is a finite morphism that is an isomorphism at points of
 % codimension $\leq 1$, then $\pi_*\omega_Y\simeq \omega_X$. \marginpar{lem/ref/expl?}
 %% .  As we discuss next, the latter needs a local
 %% version of the above argument, but the local cohomology version of
 %% \autoref{eq:33}
 %% does not seem strong enough. Instead, the proof, given in
 %% \autoref{thm:main-artin},
 %% uses the techniques of \cite{KK20}.
%\sideremark{\tiny S,03/09/23: I assumed that Thm ?? was \autoref{thm:main-artin}}%

 In order to prove \autoref{codim.3.artin.thm}, we use the techniques of \cite{KK20},
 and establish the following local, \DB version (see \autoref{rem:DB}).
%%of \autoref{codim.3.artin.thm}.

\begin{thm}%[(=~\autoref{cor:main-artin})]
  \label{thm:main-artin}%
  Let $B$ be a  local scheme  over a %(n algebraically closed)
  field of characteristic $0$, 
  and $f:(X,x)\to B$  a  local
  morphism  that is   essentially of finite type.
  Let $X_{k}$ be the fiber of $f$ over the closed point of $B$,
    $Z\subseteq X_{k}$  a closed subset of codimension $\geq 3$,
  and set $\jmath:U_{k}\leteq X_{k}\setminus Z\into X_{k}$. Assume that
  \begin{enumfull}
  \item $f$ is flat along $U_{k}$, and
  \item $\Spec \jmath_*\sO_{U_{k}}$ is \DB.
  \end{enumfull}
  Then 
  $\omega_{X/B}$ is flat over $B$ and commutes with arbitrary base
  change. 
\end{thm}

%% \begin{thm}%[(=~\autoref{cor:main-artin})]
%%   \label{thm:main-artin}%
%%   Let $(S,\frm,k)$ be a  local ring  over a %(n algebraically closed)
%%   field of characteristic $0$, 
%%   and $f:(X,x)\to (\Spec S,\frm)$  a  local
%%   morphism  that is   essentially of finite type.
%%   Let $X_{k}$ be the fiber of $f$ over the closed point of $\Spec S$,
%%   %%set $n=\dim X_k$. 
%%   %% Fix $t\in\bN, t>0$, let
%%   $Z\subseteq X_{k}$  a closed subset of codimension $\geq 3$,
%%   and set $\jmath:U_{k}\leteq X_{k}\setminus Z\into X_{k}$. Assume that
%%   \begin{enumfull}
%%   \item $f$ is flat along $U_{k}$, and
%%   \item $\Spec \jmath_*\sO_{U_{k}}$ is \DB.
%%   \end{enumfull}
%%   Then %for each $i>n-t$, $\sfh^{-i}(\dcx {X/S})$ is flat over $\Spec S$.
%%   % and commutes with any base change to a closed subscheme of $\Spec S$.
%%   % In particular,
%%   $\omega_{X/S}$ is flat over $\Spec S$ and commutes with arbitrary base
%%   change. % to a closed subscheme of $\Spec S$.
%% \end{thm}

%\noin
\autoref{thm:main-artin} will be proved as a combination of \autoref{thm:key} and
\autoref{thm:DB-to-llc}. 

As before, the method does not seem to imply that $X$ is flat over $\Spec S$, not
even if we assume that $\depth_ZX\geq 2$ as in \autoref{item:15}.  However, we do not
have a counterexample.

Note that, without the \DB assumption, such examples are easy to get:

\begin{example} 
  % Let $\{C_i:i\in I\}$ be smooth, projective curves. Fix $d_i>0$ such that
  % $d_i\leq \deg \omega_{C_i}$ for some $i\in I$, and
  % $d_j>\deg \omega_{C_j}$ for some $j\in I$.  Set $Y:=\times_i C_i$,
  % $B:= \times_i \pic^{d_i}(C_i)$, and let
  % $\pi:Y\times B\to B$ be the projection.  Consider the universal line bundle
  % $L=\boxtimes_i L_i$ on $Y\times B$, where $\deg L_i=d_i$.
  %
  % The key point is that the $h^0\bigl(Y, \omega_Y\otimes L_b^m\bigr)$ are
  % independent
  % of $b\in B$, but $h^0\bigl(Y, L_b\bigr)$ does depend on $b\in B$.  Thus the
  % universal cone
  % \[
  %   X:=C_a(Y\times B, L):=\spec_B \oplus_{m\geq 0} \pi_*L^m
  % \]
  % is not flat over $B$, but the relative $\omega_{X/B}$ is flat over $B$.
  %
  % 
  Let $\{C_i:i\in I\}$ be a finite set of smooth, projective curves.  Fix $d_i>0$
  such that $d_i\leq \deg \omega_{C_i}$ for some $i\in I$, and
  $d_j>\deg \omega_{C_j}$ for some $j\in I$.  Set $Y:=\times_i C_i$ and consider a
  line bundle $L=\boxtimes_i L_i$ on $Y$, where $\deg L_i=d_i$.

  The affine cone over $Y$ with conormal bundle $L$ (cf.\cite[3.8]{SingBook}) is
  \[
  C_a(Y, L):=\spec_k \oplus_{m\in \bZ} H^0(Y, L^m).
  \]
  By the $i=0$ case of \cite[3.13.2]{SingBook} its dualizing sheaf is the
  sheafification of the module
  \[
  \oplus_{m\in \bZ} H^0(Y, \omega_Y\otimes L^m).
  \]
  The $m$th graded pieces are
  \[
  \otimes_{i\in I} H^0(C_i, L_i^m) \qtq{and}
  \otimes_{i\in I} H^0(C_i, \omega_{C_i}\otimes L_i^m).
  \]
  Note that if $d_i\leq  \deg \omega_{C_i}$ then $h^0(C_i, L_i)$ depends on
the choice of $L_i$, not only on $\deg L_i$.

  By contrast, we claim that  $h^0(Y, \omega_Y\otimes L^m)$ depends only on
the degrees of the $L_i$ and $m$.
  Indeed, if $m\leq -1$ then $\omega_{C_j}\otimes L_j^m$ has negative
degree, so
  $H^0(Y, \omega_Y\otimes L^m)=0$.
  If $m=0$ then there is no dependence on the $L_i$, and for $m\geq 1$
  \[
  h^0(C_i, \omega_{C_i}\otimes L_i^m)=m\deg L_i+g(C_i)-1.
  \]

  Now set $B:= \times_i \pic^{d_i}(C_i)$ and note that
  $Y\times B\simeq \times_i\left(C_i\times\pic^{d_i}(C_i)\right)$.  Let $P_i$ denote
  the universal degree $d_i$ line bundle on $C_i\times\pic^{d_i}(C_i)$ and let
  $P=\boxtimes P_i$ on $Y\times B$. Further let $\pi:Y\times B\to B$ be the
  projection, and consider the universal cone
   \[
     X_B:=C_a(Y\times B, P):=\spec_B \oplus_{m\geq 0} \pi_*P^m
   \]
   over $B$.  As we noted, the $h^0\bigl(Y, \omega_Y\otimes P_b^m\bigr)$ are
   independent of $b\in B$, so the dualizing sheaf of $X_B$ is flat over $B$.
   However, $h^0\bigl(Y, P_b\bigr)$ does depend on $b\in B$, thus the structure sheaf
   is not flat over $B$.
   Note that $h^1\bigl(Y, P_b\bigr)$ also depends on $b\in B$, and when
   $h^1\bigl(Y, P_b\bigr)\neq 0$, then $C_a(Y,P_b)$, the normalization of the fiber
   of $X_B$ over $b$, is not \DB by \cite[2.5]{GK14}.
   %% universal cone
   %% \[
   %%   X:=C_a(Y\times B, L):=\spec_B \oplus_{m\geq 0} \pi_*L^m
   %% \]
   %% is not flat over $B$, but the relative $\omega_{X/B}$ is flat over $B$.
\end{example}

We also prove that KSB stability is automatic in codimension $3$ in a different
manner, namely that it is enough to check it on general hyperplane sections.

\begin{cor}\label{codim.3.artin.thm.c}
  Let $g:X\to B$ be a quasi-projective morphism of pure relative dimension $n\geq 3$
  over a field of characteristic 0, and $\Delta=\sum a_iD_i$, where the $D_i$ are
  relative Mumford $\bZ$-divisors.  Assume that
  \begin{enumfull}%[label=(\ref{codim.3.artin.thm}.\arabic*)]
  \item\label{item:23.c} $a_i\in \left\{\frac12, \frac23, \frac34, \dots, 1\right\}$
    for every $i$,
    %% \item\label{item:13} $\codim (Z_b\subset X_b)\geq 3$ for every $b\in B$,
    %% \item\label{item:14} $g\resto U: U\to B$ is flat and the fibers
    %%   $(U_b, \Delta|_{U_b})$ are semi-log-canonical,
    %% \item\label{item:14b}$\omega_{U/B}^{[m]}\bigl(\tsum_i
    %%   \rdown{ma_i}D_i|_U\bigr)$ is
    %%   flat over $B$ and commutes with base changes for every $m\in \bZ$,
    %%   $g\resto U:(U, \Delta|_U)\to B$ is locally KSB stable,
  \item\label{item:15.c} $\depth_xX\geq \min\bigl\{2, \codim (x, g^{-1}(g(x))\bigr\}$
    for every $x\in X$,
  \item\label{item:16.c} the normalization $(\ol X_b, \ol C_b+\ol\Delta_b)\to X_b$ is
    log canonical for every $b\in B$, and
  \item\label{item:16.c} general relative surface sections of $(X, \Delta)\to B$ are
    locally KSB stable.
\end{enumfull}
Then  \sfref{item:14con}--\sfref{item:14bcon} hold.
%% \begin{enumfull}[resume]%\setcounter{enumfulli}{6}
%% \item\label{item:14con} $g: X\to B$ is flat,
%% \item\label{item:14acon} the fibers $(X_b, \Delta_b)$ are semi-log-canonical, and
%% \item\label{item:14bcon} $\omega_{X/B}^{[m]}\bigl(\tsum_i \rdown{ma_i}D_i\bigr)$
%%   is flat over $B$ and commutes with base changes for every $m\in \bZ$.
%% \end{enumfull}
\end{cor}

\begin{proof} By \cite[9.17]{ModBook} we may assume that $B$ is Artinian.  Then the
  relative pluricanonical sheaves
  $\omega_{X/B}^{[m]}\bigl(\tsum_i \rdown{ma_i}D_i\bigr)$ are $S_2$. This continues
  to hold after first tensoring with line bundles and then restricting to general
  surface sections $Y:=H_1\cap\cdots H_{n-2}\subset X$; for the latter see
  \cite[10.18]{ModBook}. Thus
  \[
  \omega_{Y/B}^{[m]}\bigl(\tsum_i \rdown{ma_i}D_i|_Y\bigr)\simeq 
  \omega_{X/B}^{[m]}\bigl(\tsum_j H_j+\tsum_i \rdown{ma_i}D_i\bigr)|_Y.
  \]
  Now by \cite[p.177]{MR1011461} or \cite[10.56]{ModBook}, the
  $\omega_{X/B}^{[m]}\bigl(\tsum_i \rdown{ma_i}D_i\bigr)$ are flat over $B$ outside a
  subset of codimesion $\geq 3$. Thus they are flat everywhere by
  \autoref{codim.3.artin.thm}.
  Over Artin rings,  flat modules are free \cite[\href{https://stacks.math.columbia.edu/tag/051G}{Tag 051G}]{stacks-project}, so  commuting with base change holds; see also \cite[9.17]{ModBook}.
\end{proof}

%% (For the precise statement see \autoref{thm:inverse-Bertini} and
%% \autoref{cor:surface-section}).

%% \begin{thm}%[(=\autoref{cor:surface-section})]
%%   \label{thm:surface-section}
%%   Let $g:X\to B$ be a morphism of finite type and of pure relative dimension at least
%%   $3$ over a field of characteristic 0, and $\Delta=\sum a_iD_i$, where the $D_i$ are
%%   relative Mumford $\bZ$-divisors. Assume that $X$ is $S_2$ and $g$ satisfies the
%%   assumptions \autoref{item:23} and \autoref{item:16}. Then if a general relative
%%   surface section of $g$ is locally KSB stable, then so is $g$.
%% \end{thm}

\begin{rem}\label{rem:DB}
  %%We have referred to the \DB property a few times already.  Unfortunately, t
  The precise definition of \DB\sings, introduced by Steenbrink \cite{Steenbrink83},
  is quite involved. It starts with the construction of the Du~Bois complex, see
  \cite{DuBois81,GNPP88}, which has a natural filtration and agrees with the usual
  de~Rham complex if $X$ is nonsingular.  For our purposes the important part is the
  $0^\text{th}$ associated graded Du~Bois complex of $X$, which is denoted by
  $\dbcx X$. This comes with a natural morphism $\sO_X\to\dbcx X$, and a separated
  scheme of finite type over $\bC$ is said to have \DB\sings if this natural morphism
  is a quasi-isomorphism.
  %% A scheme of finite type over a field of characteristic $0$ is
  %% said to have \emph{\DB\sings} if its base extension to $\bC$ does.
  For more details on the definition of \DB\sings and their relevance to higher
  dimensional geometry see \cite[Chap.6]{SingBook}.

  % In characteristic 0, the optimal setting for deformation invariance of cohomology
  % seems to be the class of \DB singularities

  As we already mentioned in \autoref{eq:33},
  %% an extremely important consequence of the definition is that
  for a proper complex variety $X$ with \DB singularities, the natural morphism
  \begin{equation}
    \label{eq:star}
    \xymatrix{%
      H^i(X^{\rm an},\bC)\ar@{->>}[r] & H^i(X^{\rm an},\sO_{X^{\rm an}})
    }
    % \tag{$\star$}
  \end{equation}
  is surjective. (At least heuristically, one may think of \DB singularities as the
  largest class for which this holds, cf.~\cite{Kovacs11d}.)

  The surjectivity in \autoref{eq:star} enables one to use topological arguments to
  control the sheaf cohomology groups $H^i(X,\sO_X)$.
  % in flat families as in \cite{MR0376678}.
  It is a key element of Kodaira type vanishing theorems \cite{Kollar87b},
  \cite[Sec.12]{Kollar95s},\cite{Kovacs00c},\cite{MR2646306} and leads to various
  results on deformations of \DB schemes \cite{MR0376678,KK10,KS13}.

  %% The one shortcoming of this property is that it requires properness. Therefore,
  %% it is natural to seek a local analog.
  The obvious candidate for a local analog of \autoref{eq:33} is the map on local
  cohomologies
  \begin{equation}
    \label{eq:33.1.jk}
    H^i_x(X^{\rm an}, \bC)\to H^i_x(X^{\rm an}, \sO_{X^{\rm an}}).
  \end{equation}
  However, this map is never surjective for $i=\dim X$. In fact, if $X$ is smooth of
  dimension $n\geq 2$, then $H^n_x(X^{\rm an}, \bC)$ is trivial, but
  $H^n_x(X^{\rm an}, \sO_{X^{\rm an}})$ is infinite dimensional.

  %% Re-examining the proper case we see that there are other natural maps involved
  %% which collapse due to the \DB assumption, but in the local version some of these
  %% provide a better choice.
  %% In general, one has

  To get the right notion, one should look at the natural morphisms
  \begin{equation}
    \label{eq:34}
    \xymatrix{%
      \bC_{X^{\rm an}} \ar[r]^-\sigma & \sO_{X^{\rm an}} \ar[r]^\varrho & \Om^0_{X^{\rm an}}
    }
  \end{equation}
  The general theory implies that the composition $\varrho\circ\sigma$ induces
  surjectivity on  (hyper)cohomology
  %% $\myR^i\Gamma$
  for any proper $X$. If $X$ has \DB\sings, then $\varrho$ is
  a quasi-isomorphism, and  the surjectivity in \autoref{eq:star} follows.

  %% However, it
  %% turns out that a very useful consequence of the surjectivity induced by
  %% $\varrho\circ\sigma$ also induces surjectivity for $\varrho$:
  %% \begin{equation}
  %%   \label{eq:35}
  %%       %%   \xymatrix{%
  %%   H^i(X^{\rm an},\sO_{X^{\rm an}})\ar@{->>}[r] & \bH^i(X^{\rm an},\Om^0_{X^{\rm
  %%   an}}). 
  %%       %%   }    
  %% \end{equation}
  %% Note that this holds without the \DB assumption. Furthermore, since this only
  %% involves sheaf cohomology, we may leave the realm of analytic spaces and work
  %% entirely algebraically.  In addition, this surjectivity remains true after
  %% localization, see \autoref{lem:Matlis}.

  Note that $\varrho$ may be represented by a map between coherent sheaves, thus it
  is possible to work with $\varrho$ entirely algebraically.  Eventually, this
  suggests that the correct local replacement of \autoref{eq:33} is the (a priori
  stronger) quasi-isomorphism of $\varrho$; see also \cite[Lemma~2.2]{Kovacs99}. This
  turns out to be equivalent to the local \DB isomorphisms
  \begin{equation}
    \label{eq:33.2.jk}
    H^i_x(X,\sO_X)\simeq \bH^i_x(X,\Om_X^0)\qtq{for  $i\in \bN$ and $x\in X$.}
  \end{equation}
  At the end this leads to the local cohomology lifting property, the key technical
  ingredient in \cite{KK20}; see \autoref{def:liftable-cohom}.
\end{rem}

\begin{notation}
  %% Here, above, and later, %in \autoref{sec:db-sings-liftable},
  %% \noin {\it Notation.}
  $\bH^i$ stands for $\myR^i\Gamma$, the $i^\text{th}$ derived functor of $\Gamma$,
  the functor of sections, and $\bH^i_x$ stands for $\myR^i\Gamma_x$, the
  $i^\text{th}$ derived functor $\Gamma_x$, the functor of sections with support at
  $x$, i.e., the $i^\text{th}$ local cohomology functor with support at $x$ on the
  derived category of quasi-coherent sheaves on $X$.
\end{notation}

\section{Filtrations on modules over Artinian local rings}\label{sec:filtrations}
\noin We recall the following notation from \cite{KK20}.

\begin{demo-r}{Maximal filtrations}\label{notation}
  Let $(S,\frm, k)$ be an Artinian local ring and $N$ a finite $S$-module with a
  filtration
  $N= N_0\supsetneq N_1\supsetneq \dots \supsetneq N_{q} \supsetneq N_{q+1}=0$ such
  that $\factor {N_{j}}{N_{j+1}} \simeq k$ as $S$-modules for each $j=0,\dots,q$.
  Further let $f:(X,x)\to (\Spec S,\frm)$ be a local morphism and denote the fiber of
  $f$ over $\frm$ by $X_{k}$. It follows that then for each $j=0,\dots,q$,
  \begin{equation}
    \label{eq:4}
    %f^*N_{j}\otimes \sO_{X_{k}} \simeq 
    f^*\left( \factor {N_{j}}{N_{j+1}} \right) \simeq \sO_{X_{k}}.
  \end{equation}
\end{demo-r}
\begin{demo-r}{Filtering $S$}\label{filt-S}
  In particular, considering $S$ as a module over itself, we choose a filtration of
  $S$ by ideals
  $S=I_0\supsetneq I_1\supsetneq \dots \supsetneq I_{q} \supsetneq I_{q+1}=0$ such that
  $\factor {I_{j}}{I_{j+1}} \simeq k$ as $S$-modules for all $0\leq j\leq q$.
  Observe that in this case $I_{1}=\frm$ and for every $j$ there exists a
  $t_j\in I_j$ such that the composition
  $\xymatrix{S\ar[r]^-{t_j \cdot} & I_j \ar[r] & \factor {I_j}{I_{j+1}}}$ induces an
  isomorphism $\factor S\frm\simeq \factor {I_j}{I_{j+1}}$. In particular,
  $\ann\left(\factor {I_j}{I_{j+1}}\right) =\frm$. Finally, let
  $S_j:=\factor S{I_j}$. Note that $S_1=\factor S\frm$ and $S_{q+1}=S$.
\end{demo-r}
\begin{demo-r}{Filtering $\omega_S$}\label{filt-o}
  % Using the notation from \eqref{filt-S}, 
  Applying Grothendieck duality to the closed embedding given by the surjection
  $S\onto S_j$ implies that $\omega_{S_j} \simeq \Hom_S(S_j, \omega_S)$ and we obtain
  injective $S$-module homomorphisms
  $\varsigma_j: \omega_{S_j} \into \omega_{S_{j+1}}$ induced by the natural
  surjection $S_{j+1}\onto S_j$.  Using the fact that the canonical module of an
  Arinian local ring, in particular $\omega_S$, is an injective module and applying
  the functor $\Hom_S(\blank, \omega_S)$ to the short exact sequence of $S$-modules
  \[
  \xymatrix{%
    0 \ar[r] & \factor{I_j}{I_{j+1}} \ar[r] & S_{j+1} \ar[r] & S_j \ar[r] & 0,
  }
  \]
  we obtain another short exact sequence of $S$-modules:
  \begin{equation}
    \label{eq:10}
    \xymatrix{%
      0 \ar[r] & \omega_{S_j}
      \ar[r]^-{\varsigma_j} & \omega_{S_{j+1}} \ar[r] & \Hom_S\left(k,
        %% \factor{I_j}{I_{j+1}}, 
        \omega_S\right )\simeq k \ar[r] & 0.
    }
  \end{equation}
  Therefore we obtain a filtration of $N=\omega_S$ by the submodules
  $N_j:=\omega_{S_{q+1-j}}$ as in \eqref{notation} where
  $q+1=\length_S(S)=\length_S(\omega_S)$.  The composition of the embeddings in
  \autoref{eq:10} will be denoted by
  $\varsigma:=\varsigma_{q}\circ\dots\circ\varsigma_1 :\omega_{S_1}\into
  \omega_{S_{q+1}}= \omega_{S}$.
\end{demo-r}

Recall that the \emph{socle} of a module $M$ over a local ring $(S,\frm,k)$ is
\begin{equation}
  \label{eq:16}
  \Soc M :=(0:\frm)_M = \{x\in M \mid \frm\cdot x =0 \} \simeq \Hom_S(k, M).
\end{equation}
$\Soc M$ is naturally a $k$-vector space and $\dim_k\Soc\omega_S=1$ by the
definition of the canonical module. In particular, $\Soc\omega_S\simeq k$ and this is
the only $S$-submodule of $\omega_S$ isomorphic to $k$.

Let us recall \cite[Lemma~3.4]{KK20}, which will be important later:

\begin{lem}\label{lem:I-times-omega}
  Using the notation from \eqref{filt-S} and \eqref{filt-o}, we have that 
  %for every $j$
  \begin{equation}
    \label{eq:12}
    \im \varsigma = \Soc \omega_{S} = I_{q} \omega_{S}. \qedhere
  \end{equation}
\end{lem}

\begin{subrem}
  Note that this is not simply stating that these modules in \autoref{eq:12} are
  isomorphic, but that they are equal as submodules of $\omega_{S}$.
\end{subrem}

\section{Families over Artinian local rings}\label{sec:db-families-over}
%Invariance of cohomology for DB morphisms
%\input{fam.tex}

\noin
We will frequently use the following notation.

\begin{notation}\label{not:top-loc-cohs}
  Let $A$ be a noetherian ring, $(R,\frm)$ a noetherian local $A$-algebra,
  $I\subset R$ a nilpotent ideal and $(T,\frn)\leteq (R/I,\frm/I)$ with natural
  morphism $\alpha:R\onto T$.
  % and $\frn\leteq \frm/I$.
\end{notation}

\begin{defini}\label{def:liftable-cohom}
  Let $A$ be a noetherian ring, and $(T,\mathfrak n)$ a noetherian local
  $A$-algebra, and $i\in\bN$ fixed. We say that $T$ has \emph{\toploccohsi over $A$}
  if for any noetherian local $A$-algebra $(R,\frm)$ and nilpotent ideal $I\subset R$
  such that $R/I\simeq T$, the natural morphism on local cohomology
  \[
  \xymatrix{%
    H^i_{\frm}(R) \ar@{->>}[r] & H^i_{\mathfrak n}(T) }
  \]
  is surjective.
  % If $T$ has \toploccohsi over $A$ for every $i\geq\dim T-t$, then we
  % say that $T$ has \emph{\toploccohst over $A$}.
  Finally, if $T$ has \toploccohsi over $A$ for every $i\in\bN$, then we say that $T$
  has \emph{\toploccohs over $A$} \cite{KK20}.
  
  We say that $T$ has \emph{\toploccohsi}, % resp.~\emph{\toploccohst},
  resp.~\emph{\toploccohs} if it has the relevant property over $\bZ$.
\end{defini}

\begin{rem}\label{rem:inheriting-top-loc-cohs}
  Notice that using the above notation, if $\phi:A'\to A$ is a ring homomorphism
  from another noetherian ring $A'$ then if $T$ has \toploccohsi over $A'$, then it
  also has \toploccohsi over $A$. In particular, if $T$ has \toploccohsi over $\bZ$,
  then it has \toploccohsi over any noetherian ring $A$ justifying the above
  terminology.

  Furthermore, if $A=k$ is a field of characteristic $0$ then the notions of having
  \toploccohsi over $k$ and over $\bZ$ are equivalent. This follows in one direction
  by the above and in the other by the Cohen structure theorem
  \cite[\href{https://stacks.math.columbia.edu/tag/032A}{Tag 032A}]{stacks-project}.
\end{rem}

\begin{defini}\label{def:top-loc-cohs-for-schemes}
  We extend this definition to schemes: Let $(X,x)$ be a local scheme over a
  noetherian ring $A$. Then we say that $(X,x)$ has \emph{\toploccohsi over $A$} if
  $\sO_{X,x}$ has \toploccohsi over $A$. If $f:X\to Z$ is a morphism of % noetherian
  schemes then we say that $X$ has \emph{\toploccohsi over $Z$} if $(X,x)$ has
  \toploccohsi over $A$ for each $x\in X$ and for each $\Spec A\subseteq Z$ open
  affine neighbourhood of $f(x)\in Z$. This also extends the notion of
  % \emph{\toploccohst} and
  \emph{\toploccohs} in the obvious way.
\end{defini}

\begin{lem}\label{rem:top-loc-cohs-is-hereditary}
  Let $X\to Y\to Z$ be morphisms schemes. If $X$ has \toploccohsi over $Z$, then $X$
  has \toploccohsi over $Y$ as well.

  In particular, if $X$ has \toploccohsi over a field $k$, then it has \toploccohsi
  over any other $k$-scheme to which it admits a morphism. In addition if $\kar k=0$,
  then $X$ has \toploccohsi.
\end{lem}

\begin{proof}
  This follows from the definitions and \autoref{rem:inheriting-top-loc-cohs}.
\end{proof}

% \begin{rem}
%   A simple consequence of the definition is that if $X$ has \toploccohs over a scheme
%   $Z$, then for any morphism $g:X\to B$ of $Z$-schemes $X$ has \toploccohs over $B$
%   as well.
% \end{rem}

\noin Let us recall the following simple lemma %regarding \toploccohs
from \cite[Lemma~4.4]{KK20}:

\begin{lem}\label{lem:surj-for-R-implies-surj-for-M}\label{loc-coh-surj-2}
  Using \autoref{not:top-loc-cohs} let $M$ be an $R$-module such that there exists a
  surjective $R$-module homomorphism $\phi: M\onto T$. Assume that the induced
  natural homomorphism $H^i_{\frm}(R) \onto H^i_\frn(T)$ is surjective for some
  $i\in\bN$.  Then the induced homomorphism on local cohomology
  \begin{equation}
    \label{eq:32}
    \xymatrix{%
      H^i_{\frm}(M) \ar@{->>}[r] & H^i_{\frm}(T)\simeq H^i_\frn(T) }
  \end{equation}
  is surjective for the same $i$. In particular, if $(T,\frn)$ has \toploccohs over
  $A$, then the homomorphism in \autoref{eq:32} is surjective for every $i\in\bN$.
\end{lem}

% \begin{proof}
%   Let $t\in M$ be such that $\phi(t)=1\in T$ and let $\beta: R\to M$ be defined by
%   $1\mapsto t$.  Then $\phi\circ\beta=\alpha :R \to T$ is the natural quotient
%   morphism, hence the surjective morphism $H^i_{\frm}(R) \onto H^i_{\frm}(T)$ factors
%   through $H^i_{\frm}(M)$ which proves the statement.
% \end{proof}

\noin
We will also need the following. 

\begin{lem}%{Grothendieck vanishing in derived categories}%
  \label{Groth-van-in-der-cats}
  Let $\mcD_i$ be the derived category of an abelian category $\mcA_i$ for $i=1,2$,
  $\Phi:\mcD_1\to\mcD_2$ a triangulated functor, and define
  $\Phi^i\leteq h^i\circ\Phi:\mcD_1\to\mcA_2$.  Let $\sfA\in\Ob\mcD_1$ such that
  $h^j(\sfA)=0$ for $j>d$ for some $d\in\bZ$ and assume that there exists an
  $m\in\bN$ such that $\Phi^i(h^j(\sfA))=0$ for $i>m$ and for each $j\in\bZ$.  Then
  $\Phi^i(\sfA)=0$ for $i>m+d$.
\end{lem}

\begin{proof}
  Consider the conjugate spectral sequence associated to $\sfA$ and $\Phi$:
  \[
    E^{p,q}_2=\Phi^p(h^q(\sfA))\Rightarrow \Phi^{p+q}(\sfA).
  \]
  By the assumptions $E^{p,q}_2=0$ if either $p>m$ or $q>d$, which implies that
  $E^{p,q}_2=0$ for $p+q>m+d$. This implies the desired statement.
\end{proof}

\begin{defini}\label{def:flat-in-codim-t}
  Let $f:X\to Y$ be a morphism. Then $f$ is said to be \emph{flat in codimension $t$}
  if there exists a closed subset $Z\subseteq X$ such that
  $\codim(Z\cap X_y,X_y)\geq t+1$ for every $y\in Y$ and $f\resto{X\setminus Z}$ is
  flat.
\end{defini}

In the proof of the next statement we will use the \emph{canonical truncation} of
cochain complexes of objects of an abelian category, which has the property that its
cohomology objects are the same as the original complex up to or above the given
index. We follow the notation and terminology of
\cite[\href{https://stacks.math.columbia.edu/tag/0118}{Tag 0118}]{stacks-project}. In
particular, for any complex $\cmx C$ and any $r\in\bZ$, we have the following
distinguished triangle of complexes,%
% \footnote{{\bf Comment, to be removed from final version:} The current definition of
%   the canonical truncations on the stacks project does not actually work. Hartshorne
%   defines these correctly in \cite{RD}, but he uses the opposite notation than the
%   stacks project. I wanted to use the stacks project notation as potentially that's
%   what most people would use. I made a comment on the site and will possibly write to
%   Johan to try to convince him to change it if needed. If he does not, then I will
%   change the reference. (The issue is that the current definition on the stacks
%   project for the ``upper'' canonical truncation makes sense, but as is it does not
%   fit into a short exact sequence as in \autoref{eq:3}, which would be part of the
%   point.)}
\begin{equation}
  \label{eq:3}
  \xymatrix{%
    % 0\ar[r] &
    \tau_{\leq r}(\cmx C) \ar[r] & \cmx C \ar[r] & \tau_{\geq r+1}(\cmx C)
    \ar[r]^-{+1} 
    &  }
\end{equation}

\begin{cor}\label{cor:coh-of-Lf-same-as-f}
  Let $(S,\frm,k)$ be an Artinian local ring, $N$ a finite $S$-module, $(X,x)$ a
  local  scheme of dimension $n$, and $f:(X,x)\to (\Spec S,\frm)$ a local
  morphism. Assume that $f$ is flat in codimension $t-1$.  Then the natural morphism
  \[
    \xymatrix{%
      \myR^i\Gamma_x(\myL f^*N) \ar[r]^-\simeq & H^i_x(f^*N) 
    }
  \]
  is an isomorphism for $i\geq n-t$.
\end{cor}

\begin{proof}
  As $f$ is flat in codimension $t-1$, it follows that $\dim\supp \myL^jf^*N\leq n-t$
  for each $j<0$. This implies that $H^i_x(\myL^jf^*N)=0$ for $i>n-t$ and $j<0$.  Let
  $\sfA\leteq \tau_{\leq-1}(\myL f^*N)$ and $\sfB\leteq \tau_{\geq 0}(\myL
  f^*N)$. Then \autoref{eq:3} gives a distinguished triangle of complexes of
  $\sO_X$-modules,
  \[
    \xymatrix{%
      % 0 \ar[r] &
      \sfA \ar[r] & \myL f^*N \ar[r] & \sfB \ar[r]^-{+1} & .  }
  \]
  Furthermore, $h^{j}(\sfA)=\myL^{j}f^*N$ for $j<0$ and $h^{j}(\sfA)=0$ for
  $j\geq 0$, hence \autoref{Groth-van-in-der-cats} (for $\sfA$, $\Phi=\myR\Gamma_x$,
  $m=n-t$ and $d=-1$) implies that $\myR^i\Gamma_x(\sfA)=0$ for $i>n-t-1$. Finally,
  $\sfB\qis f^*N$, so the desired statement follows.
\end{proof}

\begin{prop}\label{thm:loc-coh-inj}
  Let $(S,\frm,k)$ be an Artinian local ring, $(X,x)$ a local  scheme of
  dimension $n$, and $f:(X,x)\to (\Spec S,\frm)$ a local morphism. Assume that $f$ is
  flat in codimension $t-1$.  Let $N$ be a finite $S$-module with a filtration as in
  \eqref{notation} and assume that $(X_{k}, x)$, where $X_{k}$ is the fiber of
  $f$ over the closed point of $\Spec S$, has \toploccohst over $S$.  Then for each
  $i>n-t$ and for each $j$, the natural sequence of morphisms induced by the
  embeddings $N_{j+1}\into N_j$ forms a short exact sequence,
  \[
  \xymatrix{%
    0 \ar[r] & H^i_x(f^*N_{j+1}) \ar[r] & H^i_x(f^*N_{j}) \ar[r] & H^i_x\left(
      f^*\left(\factor {N_{j}}{N_{j+1}}\right) \right)\simeq
    H^i_x\left(\sO_{X_{k}} \right) \ar[r] & 0. }
  \]
\end{prop}

\begin{proof}
  Since $\ann\left(\factor {N_{j}}{N_{j+1}}\right) =\frm$, there is a natural
  surjective morphism
  \[
  f^*N_{j}\otimes \sO_{X_{k}}\onto f^*\left(\factor {N_{j}}{N_{j+1}}\right).
  \]
  %Observe that $X_{k}=X_\red$, so by %\autoref{prop:loc-coh-surj-1},
  By \autoref{loc-coh-surj-2} and \autoref{eq:4}, the natural homomorphism
  \begin{equation}
    \label{eq:5}
    \xymatrix{%
      H^i_x(f^*N_{j}) \ar@{->>}[r] & H^i_x\left(
        f^*\left(\factor {N_{j}}{N_{j+1}}\right)
      \right)\simeq 
      H^i_x\left(\sO_{X_{k}} \right) }
  \end{equation}
  is surjective for all $i\geq n-t$.
  % Since $f$ is flat,  we have a short exact sequence for every $j>0$:
  Next, consider the distinguished triangle 
  \[
    \xymatrix{%
      \myL f^*N_{j+1} \ar[r] & \myL f^*N_{j} \ar[r] & \myL f^*\left(\factor
        {N_{j}}{N_{j+1}}\right) \ar[r]^-{+1} & , }
  \]
  and the induced long exact cohomology sequence for the functor $\myR\Gamma_x$.  By
  \autoref{cor:coh-of-Lf-same-as-f} the terms of that long exact sequence maybe
  replaced by terms in the form of $H^i_x(f^*(\ \ ))$ for $i\geq n-t$ and hence the
  statement follows from \autoref{eq:5}.
\end{proof}

\begin{demo-r}{The exceptional inverse image of the structure
    sheaves}\label{uppershriek}%
  Let $(S,\frm,k)$ be an Artinian local ring with a filtration by ideals as in
  \eqref{filt-S}. Further let $f:X\to \Spec S$ be a morphism which is essentially of
  finite type and $f_j=f\resto{X_j}:X_j:=X\times_{\Spec S}\Spec S_j\to \Spec S_j$
  where $S_j=S/I_j$ as defined in \eqref{filt-S}, e.g., $X_{q+1}=X$ and
  $X_{1}=X_{k}$, the fiber of $f$ over the closed point of $S$. By a slight abuse
  of notation we will denote $\omega_{\Spec S}$ with $\omega_S$ as well, but it will
  be clear from the context which one is meant at any given time.

  Using the description of the exceptional inverse image functor via the
  residual/dualizing complexes \cite[(3.3.6)]{Conrad00} (cf.\cite[3.4(a)]{RD},
  \cite[\href{http://stacks.math.columbia.edu/tag/0E9L}{Tag 0E9L}]{stacks-project}):
  \begin{equation}
    \label{eq:21}
    f^! = \myR\sHom_X(\myL f^*\myR\sHom_S(\blank, \dcx S), \dcx X) 
  \end{equation}
  and because $S$ is Artinian, $\dcx{S_j}\simeq \omega_{S_j}$ for each $j$ and we
  have that
  \begin{equation}
    \label{eq:2}
    \dcx{X_j/S_j}\simeq f_j^!\sO_{\Spec S_j}\simeq 
    \myR\sHom_{X_j}(\myL f_j^*\omega_{S_j},    \dcx {X_j}).
  \end{equation}
\end{demo-r}

\noin
In the rest of this section we will use the following notation and assumptions. 

\begin{assume}\label{ass:DB}
  Let $(S,\frm,k)$ be an Artinian local ring, $(X,x)$ a local  scheme of
  dimension $n$, and $f:(X,x)\to (\Spec S,\frm)$ a local morphism. Assume that $f$ is
  flat in codimension $t-1$ and that $(X_{k}, x)$, where $X_{k}$ is the fiber of
  $f$ over the closed point of $\Spec S$, has \toploccohst over $S$.
  % Note that we will keep using the notation introduced in \autoref{eq:17} and
  % \autoref{eq:19} and that by definition $x\in X_{k}$.
\end{assume}

\begin{thm}\label{thm:surjectivity}
  For each $i>n-t$ and each $j\in\bN$,
  \begin{enumerate}
  \item\label{item:1} there exists a natural surjective morphism
    $\xymatrix{%
      \varrho_{i,j} : \sfh^{-i}(\dcx {X_{j+1}/S_{j+1}}) \ar@{->>}[r] & \sfh^{-i}(\dcx
      {X_{j}/S_{j}}) }$,
  \item\label{item:2} there exists a natural surjective morphism $\xymatrix{%
      \varrho^i= \varrho_{i,1}\circ\dots\circ\varrho_{i,q} : \sfh^{-i}(\dcx {X/S})
      \ar@{->>}[r] & \sfh^{-i}(\dcx {X_{k}}) }$,
  \item\label{item:3} the natural morphisms $\varrho_{i,j}$ fit into a \ses,
    \[
      \xymatrix@C4em{%
        0 \ar[r] & \sfh^{-i}(\dcx {X_{k}}) \ar[r] & \sfh^{-i}(\dcx
        {X_{j+1}/S_{j+1}}) \ar[r]^-{\varrho_{i,j}} & \sfh^{-i}(\dcx {X_j/S_j})
        \ar[r] & 0,}
    \]
  %\item\label{item:9} $\ker\sfh^{-i}(\varrho_{q})= I_{q}\sfh^{-i}(\dcx {X/S})$,
  \item\label{item:5}
    $\ker \varrho_{i,j} = I_{j}\sfh^{-i}(\dcx {X_{j+1}/S_{j+1}}) \simeq
    \factor {I_{j}\sfh^{-i}(\dcx {X/S})}{I_{j+1}\sfh^{-i}(\dcx {X/S})}$,
  \item\label{item:4}
    $\sfh^{-i}(\dcx {X_{j}/S_{j}})\simeq \factor {\sfh^{-i}(\dcx {X/S})}
    {I_{j}\sfh^{-i}(\dcx {X/S})}\simeq {\sfh^{-i}(\dcx {X/S})} \otimes_{\sO_X}
    \sO_{X_j}$, and
  \item\label{item:11} $\ker \varrho^i = \frm\sfh^{-i}(\dcx {X/S})$.
  \end{enumerate}
  % In particular, the restriction
  % \begin{equation}
  %   \label{eq:7}
  %   \xymatrix{%
  %     {\omega_ {X/S}} \ar@{->>}[r] & {\omega_ {X_{k}}} }  
  % \end{equation}
  % is surjective.
\end{thm}

\begin{proof}
  Let $N=\omega_S$ and consider the filtration on $N$ given by
  $N_{j}=\omega_{S_{q+1-j}}$, cf.~\eqref{filt-o}, \autoref{eq:10}.
  % , and \eqref{nat-morph}.
  %
  Further let % $x\in X_{k}$ % be a closed point and
  $(\ )\what{\ }$ denote the completion at $x$ (the closed point of $X$).  Then by
  \autoref{thm:loc-coh-inj}, for each $i>n-t$ and each $j$, there exists a short
  exact sequence
  \begin{equation}
    \label{eq:8}
    \xymatrix{%
      0 \ar[r] & H^i_x(f^*\omega_{S_{j}}) \ar[r] & H^i_x(f^*\omega_{S_{j+1}}) \ar[r] &
      H^i_x\left(f^*\left(\factor {\omega_{S_{j+1}}}{\omega_{S_{j}}}\right)\right)
      % \simeq H^i_x\left(\sO_{X_{k}} \right)
      \ar[r] & 0. }    
  \end{equation}
  Notice that $f^*\omega_{S_j}\simeq f_j^*\omega_{S_j}$.  Combining this observation
  for both $j$ and $j+1$ with \autoref{cor:coh-of-Lf-same-as-f} yields that this \ses
  may also be written as
  \begin{equation}
    \label{eq:1}
    \xymatrix{%
      0 \ar[r] & \myR\Gamma^i_x(\myL f_j^*\omega_{S_{j}}) \ar[r] &
      \myR\Gamma^i_x(\myL f_{j+1}^*\omega_{S_{j+1}}) \ar[r] & 
      \myR\Gamma^i_x\left(f^*\left(\factor
          {\omega_{S_{j+1}}}{\omega_{S_{j}}}\right)\right) 
      % \simeq \myR\Gamma^i_x\left(\sO_{X_{k}} \right)
      \ar[r] & 0. }        
  \end{equation}
  Applying local duality \cite[\href{https://stacks.math.columbia.edu/tag/0AAK}{Tag
    0AAK}]{stacks-project} to \autoref{eq:1} gives the short exact sequence
  \[
    \xymatrix@C1.5em{%
      0 \ar[r] & \sExt^{-i}_{X} \left(f^*\left(\factor
          {\omega_{S_{j+1}}}{\omega_{S_{j}}}\right), %\sO_{X_{k}}
        \dcx{X} \right) \what{\vphantom{\big)}} \ar[r] & \sExt^{-i}_{X}(\myL
      f_{j+1}^*\omega_{S_{j+1}}, \dcx{X} )\ \what{} \ar[r] & \sExt^{-i}_{X}(\myL
      f_j^*\omega_{S_{j}}, \dcx{X} )\ \what{} \ar[r] & 0. }
  \]
  % for all $i>n-t$ and $j\in\bN$.
  Since completion is faithfully flat
  \cite[\href{http://stacks.math.columbia.edu/tag/00MC}{Tag 00MC}]{stacks-project},
  this implies that there are short exact sequences
  \begin{align}
    \label{eq:11}
    \begin{split}
      \xymatrix{%
        0 \ar[r] & \sExt^{-i}_{X} \left(f^*\left(\factor
            {\omega_{S_{j+1}}}{\omega_{S_{j}}}\right), %\sO_{X_{k}}
          \dcx{X} \right) \ar[r] & \hskip6em& \hskip6em } \\ \xymatrix{%
        && & \ar[r] & \sExt^{-i}_{X}\left(\myL f_{j+1}^*\omega_{S_{j+1}}, \dcx{X}
        \right) \ar[r] & \sExt^{-i}_{X}\left(\myL f_j^*\omega_{S_{j}}, \dcx{X}
        \right) \ar[r] & 0. }
    \end{split}
  \end{align}\
  By Grothendieck duality
  \[
    \myR\sHom_{X}(\myL f_j^*\omega_ {S_j}, \dcx {X}) \simeq \myR\sHom_{X_j}(\myL
    f_j^*\omega_ {S_j}, \dcx {X_j}),
  \]
  and hence
  $\sExt^{-i}_{X}\left(\myL f_j^*\omega_{S_{j}}, \dcx{X} \right)\simeq \sfh^{-i}(\dcx
  {X_{j}/S_{j}})$ for each $i,j$,
  % and $\sExt^{-i}_{X} \left(\sO_{X_{k}}, \dcx{X} \right)\simeq \sfh^{-i}(\dcx
  % {X_{k}})$
  by \autoref{eq:2}. %, \autoref{eq:9}.
  % Further observe that the surjective morphism in \autoref{eq:11} is the
  % $-i^\text{th}$ cohomology sheaf of the Grothendieck dual of $f^*\varsigma_{j}$ and
  % hence via the above isomorphisms, it corresponds to
  % $\sfh^{-i}(\varrho_{j})$.
  Therefore defining $\varrho_{i,j}$ as the surjective morphism in \autoref{eq:11}
  implies \autoref{item:1}.
  Composing the surjective morphisms in \autoref{eq:11} for all
  $j$ implies that the natural morphism
  \[
    \xymatrix@C4em{%
      \sfh^{-i}(\dcx {X/S})\simeq \sExt^{-i}_{X}\left(f^*\omega_S, \dcx{X} \right)
      \ar@{->>}[r]^-{\varrho^i} & \sExt^{-i}_{X}\left(f^*\omega_{S_q}, \dcx{X}
      \right) \simeq \sfh^{-i}(\dcx {X_{k}}) }
  \]
  is surjective and hence \autoref{item:2} follows as well.

  By \autoref{eq:10}
  $f^*\left(\factor {\omega_{S_{j+1}}}{\omega_{S_{j}}}\right)\simeq \sO_{X_{k}}$,
  and hence
  $\sExt^{-i}_{X} \left(f^*\left(\factor {\omega_{S_{j+1}}}{\omega_{S_{j}}}\right),
    \dcx{X} \right)\simeq \sfh^{-i}(\dcx {X_{k}})$, so \autoref{eq:11} also implies
  \autoref{item:3}.  

  Composing the injective maps in \autoref{eq:8} for all $j$ shows that the embedding
  $\varsigma: \omega_{S_1}\into \omega_S$ induces an embedding on local cohomology:
  \begin{equation}
    \label{eq:15}
    % H^i_x\left(f^*\Soc\omega_S \right) =
    H^i_x(f^*\omega_{S_1}) \subseteq
    H^i_x(f^*\omega_{S}).
  \end{equation}
  Next we prove \autoref{item:5} for $j=q$ first.  Since
  $\sfh^{-i}(\dcx {X_{q}/S_{q}})$ is supported on $X_{q}$ it follows that
  \[
  I_{q}\sfh^{-i}(\dcx {X/S})\subseteq K:= \ker \sfh^{-i}(\varrho_{q})
  %  \left[ \sfh^{-i}(\dcx {X/S})\onto \sfh^{-i}(\dcx {X_{q}/S_{q}}) \right].
  \]
  % By \autoref{thm:surjectivity}\autoref{item:3} $K\simeq \sfh^{-i}(\dcx{X_{k}})$.
  % and hence $I_{q}K=0$
  %
  Recall from \eqref{filt-S} that there exists a $t_{q}\in I_{q}$ such that
  $I_{q}=St_{q}\simeq \factor S\frm$ and from \autoref{lem:I-times-omega} that
  $I_{q}\omega_S=\Soc \omega_S$. It follows that for $j=q$ the short exact sequence
  of \autoref{eq:10} takes the form
  \begin{equation}
    \label{eq:6}
    \xymatrix@C3em{%
      0 \ar[r] & \omega_{S_{q}} \ar[r] & \omega_{S} \ar[r]^-{\tau} &
      \Soc\omega_S \ar[r] & 0,  }
  \end{equation}
  where $\tau:\omega_S\onto\Soc\omega_S\subset \omega_S$ may be identified with
  multiplication by $t_{q}$ on $\omega_S$.  Applying $f^*$ and taking local
  cohomology we obtain the sequence
  \begin{equation}
    \label{eq:14}
    \xymatrix@C3em{%
      0 \ar[r] & H^i_x(f^*\omega_{S_{q}}) \ar[r] & H^i_x(f^*\omega_{S})
      \ar[r]^-{H^i_x(\tau)} & H^i_x\left(f^*\Soc\omega_S \right) \ar[r] & 0, }
  \end{equation}
  which coincides with \autoref{eq:8} for $j=q$, and hence it is exact.  
  % since $H^i_x\left(f^*\Soc\omega_S \right) =H^i_x(I_qf^*\omega_S)$, 
  Further note that the morphism $H^i_x(\tau)$ may also be identified with
  multiplication by $t_{q}$ on $H^i_x(f^*\omega_S)$.  By \autoref{lem:I-times-omega}
  and \autoref{eq:15}, the natural morphism
  $H^i_x(\varsigma): H^i_x\left(f^*\Soc\omega_S \right) =H^i_x(I_qf^*\omega_S) =
  H^i_x(f^*\omega_{S_1}) \to H^i_x(f^*\omega_S)$ is injective.  Since $H^i_x(\tau)$,
  i.e., multiplication by $t_q$ on $H^i_x(f^*\omega_{S})$, is surjective onto
  $H^i_x\left(f^*\Soc\omega_S \right)$, it follows that
  \begin{equation}
    \label{eq:25}
    \xymatrix@R5em{%
      H^i_x\left(f^*\Soc\omega_S \right) \ar[r]_-{H^i_x(\varsigma)}^-{\simeq} & 
      \im H^i_x(\varsigma)       = I_{q} H^i_x(f^*\omega_{S}) \ar@{^(->}[r] & 
      H^i_x(f^*\omega_{S}),
    }
  \end{equation}
  i.e., $H^i_x\left(f^*\Soc\omega_S \right)$ coincides with
  $I_{q} H^i_x(f^*\omega_{S})$ as submodules of $H^i_x(f^*\omega_{S})$.
  Next let $E$ be an injective hull of $\kappa(x)=\factor{\sO_{X,x}}{\frm_{X,x}}$ and
  consider a morphism $\phi: H^i_x(f^*\Soc\omega_S)\to E$. As $E$ is injective,
  $\phi$ extends to a morphism $\wt\phi: H^i_x(f^*\omega_S)\to E$.  If
  $a\in H^i_x(f^*\omega_S)$, then
  %% Since $H^i_x(\tau)$ is identified with multiplication  by $t_{q}$ on
  %% $H^i_x(f^*\omega_S)$ and hence 
  $t_{q}a\in I_{q}H^i_x(f^*\omega_S)=H^i_x\left(f^*\Soc\omega_S \right)$, so
  \[
  t_{q}\wt\phi(a)=\wt\phi(t_{q}a)=\phi(t_{q}a)=\left(\phi\circ
    H^i_x(\tau)\right)(a)
  \]
  Therefore, $\phi\circ H^i_x(\tau)= t_{q}\wt\phi$. 
  Similarly, if $\psi: H^i_x(f^*\omega_S)\to E$ is an arbitrary morphism, then
  setting $\phi=\psi\resto{H^i_x(f^*\Soc\omega_S)}: H^i_x(f^*\Soc\omega_S)\to E$ and
  applying the same computation as above, with $\wt\phi$ replaced by $\psi$, shows
  that $\phi\circ H^i_x(\tau)= t_{q}\psi$.
  It follows that the embedding induced by $H^i_x(\tau)$,
  \begin{equation}
    \label{eq:28}
    \alpha: \Hom_{\sO_{X,x}}(H^i_x(f^*\Soc\omega_S), E)\into
    \Hom_{\sO_{X,x}}(H^i_x(f^*\omega_S), E)
  \end{equation}
  identifies $\Hom_X(H^i_x(f^*\Soc\omega_S), E)$ with
  $I_{q}\Hom_X(H^i_x(f^*\omega_S), E)$. By local duality
  % (and the argument at the beginning of the proof)
  it follows that
  % the completion at $x$
  \[
  \left(\factor{\ker \left[ \varrho_{i,q}: \sfh^{-i}(\dcx {X/S})\onto
        \sfh^{-i}(\dcx {X_{q}/S_{q}}) \right]}{I_{q}\sfh^{-i}(\dcx {X/S})}\right)
  \otimes \what{\sO}_{X,x}=0
  \]
  % is zero 
  and hence, since completion is faithfully flat, this implies \autoref{item:5} in
  the case $j=q$.  % Applying \autoref{item:5}
  Running through the same argument with $S$ replaced by $S_{j+1}$ gives the equality
  in \autoref{item:5} for all $j$.  In addition, \autoref{item:5} for $j=q$ also
  implies \autoref{item:4} for $j\geq q$.
  Assuming that \autoref{item:4} holds for $j=r+1$ implies the isomorphism in
  \autoref{item:5} for $j=r$. In turn, the entire \autoref{item:5} for $j=r$,
  combined with \autoref{item:4} for $j=r+1$, implies \autoref{item:4} for
  $j=r$. Therefore, \autoref{item:5} and \autoref{item:4} follow by descending
  induction on $j$ and then \autoref{item:11} follows from \autoref{item:5} and the
  definition of $\varrho^i$.
\end{proof}

\noin We will also need the following simple lemma from \cite[4.11]{KK20}.

\begin{lem}\label{lem:simple}
  Let $R$ be a ring. $M$ an $R$-module, $t\in R$ and $J=(t)\subseteq R$. Assume that
  $(0:J)_M=(0:J)_R\cdot M$. Then the natural morphism
  $\xymatrix{J\otimes_R M\ar[r]^-\simeq & JM}$ is an isomorphism.
\end{lem}

% \begin{proof}
%   This natural morphism is always surjective. Suppose $m\in M$ is such that
%   $t\otimes m\mapsto 0$ via this morphism. In other words such that $tm=0$. This
%   means, by definition, that $m\in (0:J)_M$ and hence by assumption there exist
%   $y\in (0:J)_R\subseteq R$ and $m'\in M$ such that $m=ym'$. Then
%   $t\otimes m=t\otimes ym'=yt\otimes m'=0$, since $yt=0$. This proves the claim.
% \end{proof}

The the following proposition and its proof is essentially the same as that of
\cite[Prop.~4.12]{KK20}. We include it here because the original situation here is
slightly different from \cite{KK20}, although the difference in the original
situation does not influence anything in this particular proof.

\begin{prop}\label{prop:tensor}
  Using the same notation as above,
  \begin{enumerate}
  \item\label{item:10}
    $I_j\otimes \sfh^{-i}(\dcx {X/S})\simeq I_{j}\sfh^{-i}(\dcx {X/S})$,
  \item\label{item:6} for any $l\in\bN$,
    $\factor{I_j}{I_{j+l}}\otimes \sfh^{-i}(\dcx {X/S})\simeq \factor
    {I_{j}\sfh^{-i}(\dcx {X/S})}{I_{j+l}\sfh^{-i}(\dcx {X/S})}$, and
  \item\label{item:12} for any $l\in\bN$,
    $\factor{\frm^l}{\frm^{l+1}}\otimes \sfh^{-i}(\dcx {X/S})\simeq \factor
    {\frm^l\sfh^{-i}(\dcx {X/S})}{\frm^{l+1}\sfh^{-i}(\dcx {X/S})}$.
  \end{enumerate}
\end{prop}

\begin{proof}
  Notice that since $H^i_x(f^*\Soc\omega_S)$ is both a quotient and a submodule of
  $H^i_x(f^*\omega_S)$, there are two natural maps between
  $\Hom_{\sO_{X,x}}(H^i_x(f^*\Soc\omega_S), E)$ and
  $\Hom_{\sO_{X,x}}( H^i_x(f^*\omega_S), E)$.  Regarding $H^i_x(f^*\Soc\omega_S)$ a
  quotient module via $H^i_x(\tau)$ we get the embedding
  $\alpha=(\blank)\circ H^i_x(\tau)$ in \autoref{eq:28}, and considering it a
  submodule the restriction map
  \begin{equation*}
%    \label{eq:23}
%    \begin{aligned}
      \xymatrix@R0em{%
        \beta: \Hom_{\sO_{X,x}}( H^i_x(f^*\omega_S), E) \ar[r] &
        \Hom_{\sO_{X,x}}(H^i_x(f^*\Soc\omega_S), E). \\
        \phi \ar@{|->}[r] & \phi\resto{H^i_x(f^*\Soc\omega_S)} }
%    \end{aligned}
  \end{equation*}
  These maps are of course not inverses to each other. In fact, we have already
  established (cf.~\autoref{eq:28}) that
  $\phi\resto{H^i_x(f^*\Soc\omega_S)}\circ H^i_x(\tau)= t_{q}\phi$ and hence the
  composition $\alpha\circ\beta$ is simply multiplication by $t_q$:
  \begin{equation}
    \label{eq:24}
    \begin{aligned}
      \xymatrix{%
        & %
        \ar@{}[l]^(.675){}="a" %
        \ar@{} "a";[d]^(.05){}="b"%
        \ar[rd]^(.55){\alpha\circ\beta}
        \phi\in \Hom_{\sO_{X,x}}(H^i_x(f^*\omega_S), E) \ar[r]^-\beta &
        \Hom_{\sO_{X,x}}(H^i_x(f^*\Soc\omega_S), E) \ar[d]^-\alpha_-\simeq \\
        && %
        \ar@{}[l]^(.45){}="c" %
        \ar@{} "c";[u]^(.15){}="d"%
        \ar@/_1em/@{|->} "b";"c" %
        t_q \phi\in I_q\Hom_{\sO_{X,x}}(H^i_x(f^*\omega_S), E). }
    \end{aligned}
  \end{equation}
  This implies, (cf.~\autoref{eq:15} and \autoref{eq:25}), that $\varrho^i$ may be
  identified with multiplication by $t_q$ on $\sfh^{-i}(\dcx {X/S})$.  Together with
  \autoref{thm:surjectivity}\autoref{item:11} this implies that
  \[
    (0:I_q)_{\sfh^{-i}(\dcx {X/S})} =\ker\varrho^i= \frm \sfh^{-i}(\dcx {X/S}) =
    (0:I_q)_S\cdot \sfh^{-i}(\dcx {X/S}),
  \]
  and hence the natural morphism
  \begin{equation}
    \label{eq:26}
    \xymatrix{%
      I_q\otimes_S \sfh^{-i}(\dcx {X/S}) \ar[r]^-\simeq & I_q\sfh^{-i}(\dcx {X/S}) 
    }
  \end{equation}
  is an isomorphism by \autoref{lem:simple}.
  Now assume, by induction, that \autoref{item:10} holds for $S_q$ in place of
  $S$. In particular, keeping in mind that $S_q=\factor S{I_q}$, the natural map
  \begin{equation}
    \label{eq:30}
    \xymatrix{%
      \factor{I_{j}}{I_{q}}\otimes_{S_q} 
      \sfh^{-i}(\dcx {X_q/S_q})\ar[r]^-\simeq & 
      \left(\factor{I_{j}}{I_{q}}\right)\sfh^{-i}(\dcx {X_q/S_q}) 
    }
  \end{equation}
  is an isomorphism for all $j$.
  % Since $q>1$, $I_{q-1}\subseteq\frm$ and hence $I_{q-1}I_q=0$. 
  Consider the \ses (cf.~\autoref{thm:surjectivity}\autoref{item:4}),
  \begin{equation*}
    %\label{eq:31}
    \xymatrix{%
      0 \ar[r] & I_q \sfh^{-i}(\dcx {X/S}) 
      \ar[r] & \sfh^{-i}(\dcx {X/S}) \ar[r] & 
      %\factor{\sfh^{-i}(\dcx{X/S})}{I_q \sfh^{-i}(\dcx {X/S})} 
      \sfh^{-i}(\dcx {X_q/S_q})  \ar[r] &  0
    }
  \end{equation*}
  and apply $\factor{I_{j}}{I_{q}}\otimes_{S} (\blank)$.  The image of
  $\factor{I_{j}}{I_{q}}\otimes_{S} I_q \sfh^{-i}(\dcx {X/S})$ in
  $\factor{I_{j}}{I_{q}}\otimes_{S} \sfh^{-i}(\dcx {X/S})$ is $0$ and hence by
  \autoref{eq:30} the natural map
  \begin{multline*}
    \xymatrix{%
      \boxed{\factor{I_{j}}{I_{q}}\otimes_{S} \sfh^{-i}(\dcx {X/S})} \simeq
      \factor{I_{j}}{I_{q}}\otimes_{S_q} \sfh^{-i}(\dcx {X_q/S_q}) \ar[r]^-\simeq &
      \left(\factor{I_{j}}{I_{q}}\right) \sfh^{-i}(\dcx{X_q/S_q})
      \simeq }\\
    \xymatrix{%
      \ar@{}[r] & \simeq \left(\factor{I_{j}}{I_{q}}\right)
      \factor{\sfh^{-i}(\dcx{X/S})}{I_q \sfh^{-i}(\dcx {X/S})} \simeq
      \boxed{\factor{I_j\sfh^{-i}(\dcx{X/S})}{I_q \sfh^{-i}(\dcx {X/S})}}.  }
  \end{multline*}
  is an isomorphism.  This, combined with \autoref{eq:26} and the 5-lemma, implies
  \autoref{item:10}. Then \autoref{item:6} is a direct consequence of
  \autoref{item:10} and the fact that tensor product is right exact.
  
  Finally, recall, that the choice of filtration in \eqref{filt-S} was fairly
  unrestricted. In particular, we may assume that the filtration $I_\kdot$ of $S$ is
  chosen so that for all $l\in \bN$, there exists a $j(l)$ such that
  $I_{j(l)}=\frm^l$. Applying \autoref{item:6} for this filtration implies
  \autoref{item:12}.
\end{proof}

\noin The following theorem is an easy combination of the results of this section.

\begin{thm}\label{thm:key}
  Let $(S,\frm,k)$ be an Artinian local ring, $(X,x)$ a local  scheme of
  dimension $n$, and $f:(X,x)\to (\Spec S,\frm)$ a local morphism. Assume that $f$ is
  flat in codimension $t-1$ and that $(X_{k}, x)$, where $X_{k}$ is the fiber of $f$
  over the closed point of $\Spec S$, has \toploccohst over $S$.  Then for each
  $i>n-t$, $\sfh^{-i}(\dcx {X/S})$ is flat over $\Spec S$.
  % and commutes with any base change to a closed subscheme of $\Spec S$.
  In particular, if $t>0$, then $\omega_{X/S}$ is flat over $\Spec S$ and commutes
  with arbitrary base change. % to a closed subscheme of $\Spec S$.
\end{thm}

\begin{proof}
  Flatness follows from \autoref{prop:tensor}\autoref{item:12} and
  \cite[\href{http://stacks.math.columbia.edu/tag/0AS8}{Tag 0AS8}]{stacks-project}.
  % $\sfh^{-i}(\dcx {X/S})$ commutes with any base change to a closed subscheme of
  % $\Spec S$ by \autoref{thm:surjectivity}\autoref{item:4}.
  If $t>0$, then this implies that $\omega_{X/S}$ is flat over $\Spec
  S$. Furthermore, it commutes with arbitrary base change by
  \autoref{thm:surjectivity}\autoref{item:2} and \cite[9.17]{ModBook}.
\end{proof}

\section{\DB \sings and liftable local cohomology}\label{sec:db-sings-liftable}

In this section we prove a criterion for a local scheme to have \toploccohst.  As
before, $\bH^i_x$ denotes $\myR^i\Gamma_x$, the $i^\text{th}$ derived functor of
$\Gamma_x$, the functor of sections with support at $x$, i.e., the $i^\text{th}$
local cohomology functor with support at $x$ on the derived category of
quasi-coherent sheaves on $X$.
%% , which, combined with
%% \autoref{thm:key} %and \autoref{thm:DB-to-llc} proves implies \autoref{thm:main-artin}.

\begin{lem}\label{lem:Matlis}
  Let $(X,x)$ be a local scheme of dimension $n$ which is essentially of finite type
  over a %(n algebraically closed)
  field of characteristic $0$. Then $H^i_x(\sO_X)\to \bH^i_x(\Om_X^0)$ is surjective
  for each $i\in\bZ$.
\end{lem}

\begin{proof}
  This follows by applying Matlis duality to the map in \cite[Lemma~3.2]{MSS17}
  (cf.~\cite[Lemma~2.2]{Kovacs99}, \cite[Theorem~3.3]{MR3617778},
  \cite[Theorem~3.2]{KS13}, \cite[Lemma~3.3]{MSS17}).
\end{proof}

\begin{thm}\label{thm:DB-to-llc}
  Let $(X,x)$ be a local scheme of dimension $n$ which is essentially of finite type
  over a %(n algebraically closed)
  field of characteristic $0$.  Fix $t\in\bN, t>0$, and let $Z\subseteq X$ be a
  closed subset of codimension $t+2$. Further let $\sigma: Y\to X$ be
  % a finite
  an affine morphism which is
  % a homeomorphism and
  an isomorphism over $U\leteq X\setminus Z$.  Assume that $Y$ is \DB.
  % \footnote{actually we're only using that the relevant map is an isom on local
  %   cohomology. I wonder if that's equivalent to being \DB... And, we actually only
  %   need this isomorphism in the upper range, so something like ``cohomologically \DB
  %   in codimension $t$'' (which probably does not follow from ``slc in codim
  %   $t$'').}.
  %
  Then
  \begin{enumfull}
  \item\label{item:7} $H^i_x(\sO_X)\to \bH^i_x(\Om_X^0)$ is an isomorphism for
    $i\geq n-t$, and
  \item\label{item:8} $X$ has \toploccohst.
  \end{enumfull}
\end{thm}

\begin{proof}
  Let $W=\sigma^{-1}(x)\subseteq Y$ and observe that there is an equality of
  functors:
  \[
    \Gamma_{x}\circ \sigma_* = \Gamma_W.
  \]
  As $\sigma$ is an affine, morphism, $\sigma_*$ is exact, we obtain an equality of
  derived functors:
  \begin{equation}\addtocounter{equation}{2}
    \label{eq:22}
    \myR\Gamma_{x}\circ \sigma_* = \myR\Gamma_W.
  \end{equation}
  Consider the \ses
  \[
    \xymatrix{%
      0\ar[r] & \sO_X \ar[r] & \sigma_*\sO_Y \ar[r] & \sQ \ar[r] & 0,
    }
  \]
  where $\sQ$ is defined as the cokernel of the first non-zero morphism in this
  \ses.  Applying the functor $\myR\Gamma_x$, and taking into
  account \autoref{eq:22}, we obtain the following distinguished triangle:
  \[
    \xymatrix{%
      \myR\Gamma_x\sO_X\ar[r] & \myR\Gamma_W\sO_Y \ar[r] & \myR\Gamma_x\sQ
      \ar[r]^-{+1} & }
  \]

  The assumption implies that $\sQ$ is supported on $Z$, so $H^i_x(\sQ)=0$ for
  $i>n-t-2$, and hence
  \begin{equation}
    \label{eq:29}
    H^i_x(\sO_X)\simeq H^i_W(\sO_Y) \quad \text{ for $i\geq n-t$.}
  \end{equation}

  Next, consider the following diagram:
  \begin{equation*}
    %\label{eq:18}
    \begin{aligned}
      \xymatrix@C7em{%
        \sO_X \ar[d]  \ar[r] & \myR\sigma_*\sO_Y \ar[d] \\
        \Om_X^0 \ar[r] & \myR\sigma_*\Om_Y^0. }
    \end{aligned}
  \end{equation*}
  Applying $\myR\Gamma_x$ to each element %in this diagram
  and using \autoref{eq:22} and \autoref{eq:29} leads to the following:
  \begin{equation}
    \label{eq:13}
    \begin{aligned}
      \xymatrix@C7em{%
        H^i_x(\sO_X) \ar@{->>}[d] \ar[r]_-{\text{(for $i\geq n-t$)}}^-{\simeq} &
        H^i_W(\sO_Y) \ar[d]^\simeq \\
        \bH^i_x(\Om_X^0) \ar[r] & \bH^i_W(\Om_Y^0) }
    \end{aligned}
  \end{equation}
  % We proved above that
  The top horizontal arrow is an isomorphism for $i\geq n-t$ and the right vertical
  arrow is an isomorphism for all $i$, because $Y$ is \DB.  It follows that the
  diagonal map is also an isomorphism, and in particular, injective for $i\geq n-t$.
  In particular the left vertical arrow is also injective for $i\geq n-t$. It is
  surjective for each $i$ by \autoref{lem:Matlis} and hence an isomorphism for
  $i\geq n-t$. This proves \autoref{item:7}.

  Let $(R,\frm)$ be a noetherian local ring and $I\subset R$ a nilpotent ideal such
  that $R/I\simeq \sO_{X,x}$. In order to prove \autoref{item:8} we need that the
  induced natural morphism on local cohomology 
  \begin{equation}
    \label{eq:27}
        \xymatrix{%
      H^i_{\frm}(R) \ar@{->>}[r] & H^i_{x}(\sO_X) }
  \end{equation}
  is surjective for $i\geq n-t$. Let $X'\leteq \Spec R$ and consider the following
  diagram:
  \[
    \qquad\qquad\xymatrix@C5em{%
      H^i_{\frm} (R) \ar[r] \ar@{->>}[d] & H^i_x(\sO_X) \ar[d]^{\text{ (for
          $i\geq n-t$ by \ref{item:7})}}_{\simeq}  \\
      \bH^i_{\frm}(\Om_{X'}^0) \ar[r]_-\simeq & \bH^i(\Om_X^0) }
  \]
  As above, the left vertical arrow is a surjection by \autoref{lem:Matlis}. The
  bottom horizontal arrow is an isomorphism, because $X'_{\red}\simeq X_{\red}$ and
  $\Om^0$ only depends on the reduced structure by definition, 
  cf.~\cite[p.2150]{MSS17}. Finally, the right vertical arrow is an isomorphism  for
  $i\geq n-t$ by \autoref{item:7} and the combination of these implies
  \autoref{eq:27} and hence \autoref{item:8}.
\end{proof}

\begin{proof}[Proof of \autoref{thm:main-artin}]
  It follows from \autoref{thm:DB-to-llc} that the assumptions of
  \autoref{thm:main-artin} imply those of \autoref{thm:key}, which in turn implies
  the desired statement of \autoref{thm:main-artin} if $S$ is Artinian.

  If $\omega_{X/B}$ is known to commute with base changes, then one can check
  flatness over Artin subschemes of $B$ by the local criterion of flatness.
  
  The general case follows from \cite[9.17]{ModBook}, which is a variant of the local
  criterion of flatness, combined with obstruction theory.\end{proof}

% \begin{thm}\label{codim.3.delta.thm}  Let $g:X\to B$ be a morphism of
% finite type of pure relative dimension over a field of characteristic 0,
% and
% $\Delta=\sum a_iD_i$ a relatve Mumford divisor.
% Let   $Z\subset X$  be a closed subset and set $U:=X\setminus Z$.
% Assume that
% \begin{enumerate}
%    \item $a_i\in \{\frac12, \frac23, \frac34, \dots, 1\}$ for every $i$,
%    \item $\codim (Z_b\subset X_b)\geq 3$ for every $b\in B$,
%    \item $g\resto U:(U, \Delta\resto U))\to B$ is locally KSB stable,
%    \item $\depth_ZX\geq 2$, and
%    \item the normalization $(\bar X_b, \bar D_b+\bar\Delta_b)\to X_s$ is log
%      canonical for every $b\in B$.
%    \end{enumerate}
%    Then $g:(X,\Delta)\to B$ is locally KSB stable.
% \end{thm}

\begin{proof}[Proof of \autoref{codim.3.artin.thm}]
  We may assume that $B$ is a local scheme with closed point $b\in B$.  We will
  consider three, increasingly more general cases.

  \noin{\sf\small Case I}: \emph{$\Delta=0$ and $\omega_{\wt X_b}$ is locally free,}
  where $\pi:\wt X_b\to X_b$ is the demi-normalization as in
  \autoref{item:1.3.4.jk}.
          %%{codim.3.artin.thm.refs}.
  
  %% Consider a (local) base change $B'\to B$ to an Artinian local scheme and let
  %% $g'\leteq g_{B'}:X'\leteq X_{B'}\to B'$. Then the above conditions still hold over
  %% $B'$ and hence

  Note that $\omega_{X/B}$ is flat and commutes with arbitrary base change by
  \autoref{thm:main-artin}.
  %\autoref{thm:key}.
  By further localization we may assume that $\omega_{\wt X_b}$ is free. Since
  $\omega_{ X_b}\simeq \pi_*\omega_{\wt X_b}$ by \autoref{lem.1.1.jk}, we see that
  $\omega_{X/B}$ has a section $\sigma$ such that $\sigma_b$ does not vanish on
  $U_b$, hence $\sigma: \sO_{X}\to \omega_{X/B}$ is an isomorphism away from a closed
  subset $W$ for which $W_b\subset Z_b$.  In particular, $\depth_{W_b}\sO_X\geq 2$ by
  \autoref{item:15}.  Now we use the easy \cite[Lem.10.6]{ModBook} to conclude that
  $ \sO_{X}\simeq \omega_{X/B}$.  Thus $g$ is flat, $\omega_{X/B}$ is locally free,
  and so are all of its powers.

  %% $g'$ is flat.  This implies that the original by
  %% \cite[\href{https://stacks.math.columbia.edu/tag/0523}{Tag 0523}]{stacks-project}.
  %% Then it follows that $\omega_{X/B}$ is also flat and commutes with arbitrary base
  %% change by \cite[5.11]{KK20}. In particular, $\omega_{X/B}$ is locally free and all
  %% of its powers are locally free, flat, and commute with arbitrary base change.

  \noin{\sf\small Case II}: \emph{$\Delta=D$ is a $\bZ$-divisor and
    $\omega_{\wt X_b}(\wt D_b)$ is locally free}.  Note that
  $\sO_U(-D)\simeq \omega_{U/B}$ is flat over $B$ and commutes with base changes by
  assumption. Thus \autoref{glue.lem.jk} applies, and so $\omega_{X/B}(D)$ is flat
  over $B$ and commutes with base changes.

  We may assume that $\omega_{\wt X_b}(\wt D_b)$ is free with generating section
  $\wt\sigma_b$. By \autoref{lem.1.1.jk} we can identify $\wt\sigma_b$ with a section
  $\sigma_b$ of $\omega_{X_b}(D_b)$. By flatness it lifts to
  $\sigma: \sO_X\to \omega_{X/B}(D)$, which is an isomorphism over $U$.  By
  \autoref{item:15} (and the easy \cite[10.6]{ModBook}) $\sigma$ is an
  isomorphism. Thus $\omega_{X/B}(D)$ is locally free and so are its powers.

  \noin{\sf\small Case III}: \emph{The general case}.
  We may assume that $X$ is local, and
  by \cite[9.17]{ModBook} it is sufficient to prove the case when $B$ is Artinian.

  Write  $\Delta=\sum_{i\in I} a_iD_i$, where  $a_i= 1-\frac1{i}$, 
$I\subset \{2, 3,4, \dots, \infty\}$ is a finite subset and the $D_i$ are reduced divisors.

  %% $a_i\in \left\{\frac12, \frac23, \frac34, \dots\right\}$   for every $i$.
 Choose $m>0$ such that $\omega_{U_b}^{[m]}(m\Delta_b)\sim \sO_{U_b}$.
The kernel of  $\pic(U)\to \pic(U_b)$ is a $k$-vectorspace; hence divisible and torsion free.  
Thus there is a unique line bundle 
  $L_U$ on $U$ such that $L_{U_b}\sim \sO_{U_b}$ and $\omega_{U/B}^{[m]}(m\Delta)\hotimes L^{m}_U\sim \sO_U$. Let $L$ be the push-forward of $L_U$ to $X$. Take the corresponding 
  cyclic cover
  \[
    \pi: Y:=\spec_X \tsum_{j=0}^{m-1} \omega_{X/B}^{[j]}\bigl(\tsum_i
    \rdown{ja_i}D_i\bigr)\hotimes L^{[j]} \to X.
    \]
    Note that $\pi$ ramifies along the $D_i$ as follows.  If $i\geq 3$, then $\pi$
    has ramification index $i$ along $D_i$, and $\pi$ is unramified along
    $D_{\infty}$.  The $i=2$ case is somewehat special.  Then $\pi_b$ has
    ramification index $2$ along an irreducible divisor $F_b\subset X_b$ if it has
    multiplicity 1 in $D_2|_b$, and $Y_b$ is nodal along $\pi_b^{-1}(F_b)$ if $F_b$
    has multiplicity 2 in $D_2|_b$. Thus
    \[
    K_{Y_b}+\pi_b^*D_{\infty}\sim_{\bQ}
    \pi_b^*\bigl(K_{X_b}+\Delta_b\bigr).
    \]    
  In particular,  $(Y, \pi^*D_{\infty} )\to B$ satisfies the assumptions
  \autoref{item:23}-\autoref{item:16}.
  (Note that $Y\to B$ is known to be flat only over $U$, so requiring flatness only in codimension $\leq 2$ is essential here.)
  %(\ref{codim.3.artin.thm}.1--5)

  By duality, we get that
  \[
  \begin{array}{lcl}
    \pi_* \omega_{Y/B}(\pi^*D_{\infty} )&\simeq& \tsum_{j=0}^{m-1} \omega_{X/B}^{[1-j]}
    \bigl(D_{\infty}-\tsum_i \rdown{ja_i}D_i\bigr)\hotimes L^{[-j]},
    \qtq{and}\\
    (\pi_b)_* \omega_{\wt Y_b}(\pi_b^*D_{\infty} )&\simeq& \tsum_{j=0}^{m-1} \omega_{X_b}^{[1-j]}
    \bigl(D_{\infty}|_b-\tsum_i \rdown{ja_i}D_i|_b\bigr)\hotimes L^{[-j]}_b.
    \end{array}
  \]
  The $j=1$ summand of $(\pi_b)_* \omega_{\wt Y_b}(\pi_b^*D_{\infty} )$ is trivial. Thus $\omega_{\wt Y_b}(\pi_b^*D_{\infty} )$ has a section that is nowhere zero on $U_b$, so  $\omega_{\wt Y_b}(\pi_b^*D_{\infty} )$ is trivial. The previous case applies, and we conclude that all the
  \[
  \omega_{X/B}^{[1-j]}
  \bigl(D_{\infty}-\tsum_i \rdown{ja_i}D_i\bigr)\hotimes L^{[-j]}
  \]
  are flat over $B$ and commute with base changes.
  
  %% and $\omega_{Y_b}(\pi_b^*D_b )$ is locally free.  Thus $\omega_{Y/B}(\pi^*D_{\infty} )$ is
  %% locally free and commutes with base changes by the previous case.  Therefore
  %% \[
  %%   \pi_* \omega_{Y/B}(\pi^*\rdown{\Delta} )\simeq \tsum_{j=0}^{m-1} \omega_{X/B}^{[1-j]}
  %%   \bigl(\rdown{\Delta}-\tsum_i \rdown{ja_i}D_i\bigr)\hotimes L^{[-j]}
  %% \]
  %% is flat over $B$ and commutes with base changes.
  The $j=1$ summand is $L^{[-1]}$,
  whose restriction to $X_b$ is trivial.  By flatness, the constant 1 section of
  $L^{[-1]}\resto{X_b}$ lifts to a section of $L^{[-1]}$, hence $L$ is trivial.

  Now fix $0\leq r<m$ and set $1-j=r-m$. Then we get that
  \[
  \omega_{X/B}^{[r]} \bigl(D_{\infty}+\tsum_i (ma_iD_i-\rdown{(m-r+1)a_i})D_i\bigr)
  \simeq
  \omega_{X/B}^{[1-j]}
  \bigl(D_{\infty}-\tsum_i \rdown{ja_i}D_i\bigr)\hotimes L^{[-j]}
  \]
  is flat over $B$ and commutes with base changes.  Now, observe that
  \[
    \rdown{ra}+\rdown{(m-r+1)a}=
    \begin{cases}
      m+1 \qtq{if} a=1,  \qtq{and}\\
      m \qtq{if} a=\tfrac{c-1}{c} \qtq{for some} 1<c\mid m.
    \end{cases}
  \]
  This gives that
  \[
    \omega_{X/B}^{[r]} \bigl(D_{\infty}+\tsum_i (ma_i-\rdown{(m-r+1)a_i})D_i\bigr) \simeq
    \omega_{X/B}^{[r]} \bigl(\tsum_i \rdown{ra_i}D_i\bigr).
  \]
  Thus the $\omega_{X/B}^{[r]} \bigl(\tsum_i \rdown{ra_i}D_i\bigr)$ are flat over $B$
  and commute with base changes.
\end{proof}

\begin{cor} \label{4.3.cor}
  Using the notation and assumptions of \autoref{codim.3.artin.thm},
  set $D_\infty:=\sum_{i:a_i=1}D_i$. Then
  $\sO_X(-D_\infty)$ and  $\sO_{D_\infty}$ are
  flat over $B$ and commute with base changes.
\end{cor}

\begin{proof} Arguing as in {\sf\small Case III} above, we get that
  \[
  \pi_* \omega_{Y/B}\simeq \tsum_{j=0}^{m-1} \omega_{X/B}^{[1-j]}
  \bigl(-\tsum_i \rdown{ja_i}D_i\bigr)\hotimes L^{[-j]}.
  \]
  We proved that $L$ is trivial, so the $j=1$ summand is
  $\sO_X(-D_\infty)$. It is thus flat over $B$ with $S_2$ fibers.
  Therefore the induced maps  $\sO_X(-D_\infty)|_{X_b}\to\sO_{X_b}$
  are injections, hence $\sO_{D_\infty}$  is also
  flat over $B$ and commutes with base changes.
  \end{proof}
  
\section{KSBA stability}\label{sec.ksba.jk}

It is possible that the analog of \autoref{codim.3.artin.thm} holds for arbitrary
KSBA stable pairs as in \cite[Sec.8.2]{ModBook}.  Note that by \cite[7.5]{ModBook},
K-flatness of divisors is automatic in codimension $\geq 3$.  This would say that the
whole theory of KSBA stability is determined in codimension 2.

The next result is a very small step in this direction. It shows that the reduced
part of the boundary divisor behaves well in codimension $\geq 3$.

\begin{prop}\label{glue.lem.jk}
  Let $g:X\to B$ be a morphism of finite type and of pure relative dimension over a
  field of characteristic 0, $\Delta$ a relative Mumford $\bR$-divisor and
  $0\leq D\leq \Delta$ a relative Mumford $\bZ$-divisor.  Let $Z\subset X$ be a
  closed subset and set $U:=X\setminus Z$.  Assume that
  \begin{enumfull}%[label=(\ref{codim.3.artin.thm}.\arabic*)]
  \item\label{item:13.jk} $\codim (Z_b\subset X_b)\geq 3$ for every $b\in B$,
  \item\label{item:14.jk} $g\resto U: U\to B$ is flat with demi-normal fibers,
    %% $(U_b, \Delta|_{U_b})$ are semi-log-canonical,
  \item\label{item:14b.jk}$\sO_U(-D|_U)$ is
    flat over $B$ and commutes with base changes, and
    %% \item\label{item:15.jk} $\depth_ZX\geq 2$, and
  \item\label{item:16.jk} the demi-normalization $(\wt X_b, \wt\Delta_b)$ of
    $ (X_b, \Delta_b)$ is semi-log-canonical for $b\in B$.
    %%% the normalization $(\ol X_b, \ol D_b)\to X_b$ is log canonical for every
    %%% $b\in B$.
  \end{enumfull}
  Then $\omega_{X/B}(D)$ is flat over $B$ and commutes with base changes.
  %% \begin{enumfull}[resume]%\setcounter{enumfulli}{6}
  %% \item\label{item:14bcon.jk}
  %% \end{enumfull}
\end{prop}

\begin{proof}
  Take two copies $(X_i,\Delta_i)\simeq (X, \Delta)$ and
  glue them together along $D_1\simeq D_2$ to get
   \[
     g_Y:=(g_1\amalg g_2):Y:=X_1\amalg_{D_1\simeq D_2} X_2\to B.
     \]
     Let $\pi:Y\to X$ be the projection. Set $\Delta_Y:=\pi^*(\Delta-D)$ and consider
     the short exact sequence,
     \[
       \xymatrix{%
         0\ar[r] & \sO_{X_1}(-D_1)\ar[r] & \sO_Y\ar[r] & \sO_{X_2}\ar[r] & 0.  }
     \]
     As $\pi$ is finite, the push-forward of this remains exact and, using the fact
     that $\pi\resto{X_i}$ is an isomorphism, the natural morphism
     $\sO_X\to\pi_*\sO_Y$ provides a splitting of the push-forward of the above exact
     sequence. Therefore, $\pi_*\sO_Y\simeq \sO_X\oplus\sO_X(-D)$, and so
     $(Y,\Delta_Y) \to B$ is flat over $\pi^{-1}(U)$ with semi-log-canonical fibers.
     The demi-normalization of $(Y_b, \Delta_Y|_b)$ is the amalgamation of 2 copies
     of $\bigl (\wt X_b, \wt \Delta_b\bigr)$ along $\wt D_b$, hence
     semi-log-canonical. Thus $\omega_{Y/B}$ is flat over $B$ and commutes with base
     changes by \autoref{thm:main-artin}. Finally note that
     $\pi_*\omega_{Y/B}\simeq \omega_{X/B}\oplus\omega_{X/B}(D)$, thus
     $\omega_{X/B}(D)$ is flat over $B$ and commutes with base changes.
\end{proof}

%% \begin{cor}\label{K+D.thm}
%%   Let $g:(X,\Delta)\to B$ be locally KSBA stable as in \cite[Sec.8.2]{ModBook}, and
%%   $0\leq D\leq \Delta$ a relative Mumford $\bZ$-divisor. Then $\omega_{X/B}(D)$ is
%%   flat over $B$ and commutes with base changes.
%% \end{cor}

%% \begin{proof} The assumption \autoref{item:14b.jk} holds by definition
%%   \cite[8.7.4]{ModBook}  (setting $m=0$).
%% \end{proof}

\begin{rem} We claim that AFI stability, where we float all coefficients as in
  \cite[Sec.8.3]{ModBook}, is determined in codimension 2.

  To see this, note that the boundary divisor $\Delta$ is necessarily
  $\bR$-Cartier. Thus, for every point $x\in Z_b$ as in \autoref{codim.3.artin.thm},
  either $x\not\in \supp \Delta_b$, and then local stability holds by
  \autoref{codim.3.artin.thm}, or $x\in \supp \Delta_b$, and then $x$ is not an lc
  center of $X_b$.  Then $\depth_x \sO_{\wt X_b}\geq 3$ by \cite[7.20]{SingBook}
  (cf.~\cite{MR2918171} and \cite{MR2959783}), hence local stability holds by
  \cite[10.73]{ModBook}.
\end{rem}

\def\cprime{$'$} \def\cprime{$'$} \def\cprime{$'$} \def\cprime{$'$}
  \def\cprime{$'$} \def\polhk#1{\setbox0=\hbox{#1}{\ooalign{\hidewidth
  \lower1.5ex\hbox{`}\hidewidth\crcr\unhbox0}}} \def\cdprime{$''$}
  \def\cprime{$'$} \def\cprime{$'$} \def\cprime{$'$} \def\cprime{$'$}
  \def\cprime{$'$}
\providecommand{\bysame}{\leavevmode\hbox to3em{\hrulefill}\thinspace}
\providecommand{\MR}{\relax\ifhmode\unskip\space\fi MR}
\providecommand{\MRhref}[2]{%
  \href{http://www.ams.org/mathscinet-getitem?mr=#1}{#2}
}
\providecommand{\href}[2]{#2}

\end{document}